\documentclass{amsart}

\usepackage{amsmath, amssymb, amsthm, enumitem, xspace}

\newtheorem{theorem}{Theorem}[section]
\newtheorem{lemma}[theorem]{Lemma}

\newtheorem{corollary}[theorem]{Corollary}
\newtheorem{claim}[theorem]{Claim}

\newtheorem{fact}[theorem]{Fact}
\newtheorem{question}[theorem]{Question}

\theoremstyle{definition}
\newtheorem{definition}[theorem]{Definition}
\newtheorem{remark}[theorem]{Remark}

\newcommand{\cf}{\mathrm{cf}}

\newcommand{\bb}{\mathbb}

\newcommand{\otp}{\mathrm{otp}}

\newcommand{\refl}{\mathrm{Refl}}
\newcommand{\DSR}{\mathsf{DSR}}
\newcommand{\uDSR}{\mathsf{uDSR}}
\newcommand{\sDSR}{\mathsf{sDSR}}
\newcommand{\OSR}{\ensuremath{\mathsf{OSR}}\xspace}

\newcommand{\calC}{\mathcal{C}}
\newcommand{\st}{\;|\;}
\newcommand{\crit}{ {\rm crit} }
\newcommand{\forces}{\Vdash}
\newcommand{\lift}{\text{\sf{lift}}}
\newcommand{\Todorcevic}{Todor\v{c}evi\'{c}\xspace}

\title{Separating diagonal stationary reflection principles}
\author{Gunter Fuchs}
\thanks{The first author gratefully acknowledges support from the Simons foundation under award ID 580600. His research was also supported in part by PSC CUNY grant number 61567-00 49.}
\address{The College of Staten Island (CUNY)\\2800 Victory Blvd.~\\Staten Island, NY 10314}
\address{The Graduate Center (CUNY)\\365 5th Avenue, New York, NY 10016}
\email{gunter.fuchs@csi.cuny.edu}
\urladdr{www.math.csi.cuny.edu/~fuchs}
\author{Chris Lambie-Hanson}
\address{Department of Mathematics and Applied Mathematics \\
Virginia Commonwealth University \\
Richmond, VA 23284 \\ United States}
\email{cblambiehanso@vcu.edu}
\urladdr{people.vcu.edu/~cblambiehanso}


\begin{document}
\date{\today}

\begin{abstract}
We introduce three families of diagonal reflection principles for matrices of
stationary sets of ordinals. We analyze both their relationships among
themselves and their relationships with other known principles of
simultaneous stationary reflection, the strong
reflection principle, and the existence of square sequences.
\end{abstract}
\maketitle

\section{Introduction}

The study of compactness and reflection principles has been the subject of a
significant amount of set theoretic research, and the careful investigation of
the tension existing between compactness principles that arise due to
the presence of large cardinals and incompactness principles that tend to hold,
for example, in canonical inner models, has been quite fruitful. Particularly prominent
among the compactness principles that have been studied are various principles
of stationary reflection.
In this article, we investigate the relationships between different principles of
stationary reflection, particularly focusing on diagonal stationary reflection principles.
We will also prove some results closely linking these diagonal reflection principles
with certain square principles, which provide concrete instances of incompactness.
We begin by giving some background, to motivate the questions we are going to address. The notation in the upcoming definition follows \cite{SquaresScalesStationaryReflection} and
\cite{Hayut-LambieHanson:SimultaneousReflectionAndSquare}.

\begin{definition}
\label{def:StationaryReflection}
  Suppose that $\lambda$ is a regular uncountable cardinal and $S \subseteq \lambda$
  is stationary.
  \begin{enumerate}
    \item If $\alpha < \lambda$, then $S$ \emph{reflects at $\alpha$} if
    $\cf(\alpha) > \omega$ and $S \cap \alpha$ is stationary in $\alpha$. We say that
    $S$ \emph{reflects} if there is $\alpha < \lambda$ such that $S$ reflects
    at $\alpha$.
    \item If $\mathcal{T}$ is a family of stationary subsets of $\lambda$ and
    $\alpha < \lambda$, then $\mathcal{T}$ \emph{reflects simultaneously at
    $\alpha$} if $T$ reflects at $\alpha$ for every $T \in \mathcal{T}$. We
    say that $\mathcal{T}$ \emph{reflects simultaneously} if there is $\alpha
    < \lambda$ such that $\mathcal{T}$ reflects simultaneously at $\alpha$.
    \item If $1 < \kappa \leq \lambda$, then $\refl({<} \kappa, S)$ is the assertion
    that every family of fewer than $\kappa$-many stationary subsets of $S$
    reflects simultaneously. $\refl({<} \kappa^+, S)$ is typically denoted by
    $\refl(\kappa, S)$.
    \item $\refl^*({<}\kappa, S)$ is the assertion that, if $\bb{P}$ is a
    $\lambda$-directed closed forcing and $|\bb{P}| \leq \lambda$, then
    $\Vdash_{\bb{P}}``\refl({<}\kappa, S)"$.
  \end{enumerate}
\end{definition}

A major appeal of these principles is that they imply the failure of certain square principles.

\begin{definition}
\label{defn:WeakTodorcevicSquares}
Let $\lambda$ be a limit of limit ordinals.
\begin{enumerate}
  \item A sequence $\vec{\calC}=\langle\calC_\alpha\st\alpha<\lambda,\ \alpha\ \text{limit}\rangle$ is \emph{a coherent sequence of length $\lambda$} if,
  for every limit ordinal $\alpha < \lambda$,
  \begin{enumerate}
    \item $\calC_\alpha$ is a nonempty collection of clubs in $\alpha$; and
    \item for every $C \in \calC_\alpha$ and every limit point $\beta$ of $C$,
    we have $C \cap \beta \in \calC_\beta$.
  \end{enumerate}
  If, moreover, $\kappa$ is a cardinal and $|\calC_\alpha| < \kappa$ for every
  limit $\alpha < \lambda$, then $\vec{\calC}$ is said to have \emph{width} ${<}\kappa$.
  \item If $\vec{\calC}$ is a coherent sequence of length $\lambda$, then a
  \emph{thread} through $\vec{\calC}$ is a club subset $T$ of $\lambda$ that
  coheres with $\vec{\calC}$, i.e., for every limit point $\beta$ of $T$, we
  have $T \cap \beta \in \calC_\beta$.
  \item If $\kappa$ is a cardinal, then a \emph{$\square(\lambda,
  {<}\kappa)$-sequence} is a coherent sequence of length $\lambda$ and width
  ${<}\kappa$ that does not have a thread. The principle $\square(\lambda,{<}\kappa)$
  asserts that there is a $\square(\lambda,{<}\kappa)$ sequence. In place of
  $\square(\lambda,{<}\kappa^+)$, we may write $\square(\lambda,\kappa)$.
\end{enumerate}
\end{definition}

$\square(\lambda,1)$ is known as $\square(\lambda)$, and $\square(\lambda,{<}\kappa)$
becomes weaker as $\kappa$ increases.
In \cite{Hayut-LambieHanson:SimultaneousReflectionAndSquare}, Hayut and the second
author analyzed the effects of simultaneous stationary reflection on these
kinds of square principles. In what follows and throughout the paper,
if $\kappa<\beta$ where $\kappa$ is an infinite cardinal and
$\beta$ is an ordinal, then we will write
$S^\beta_\kappa$ for the set of limit ordinals less than $\beta$ of cofinality
$\kappa$. Similarly, $S^\beta_{{<}\kappa}$ denotes the set of
limit ordinals less than $\beta$ of cofinality less than $\kappa$,
and $S^\beta_{{\leq}\kappa}, S^\beta_{{\ge}\kappa}, S^\beta_{>\kappa}$
have the obvious meaning.

\begin{theorem}[{\cite[Thm.~2.13]{Hayut-LambieHanson:SimultaneousReflectionAndSquare}}] \label{hlh_thm_213}
Suppose that $\kappa<\lambda$ are cardinals, $\lambda$ is regular, and $\refl({<}\kappa,S)$ holds for some stationary $S\subseteq S^\lambda_{{\ge}\kappa}$. Then $\square(\lambda,{<}\kappa)$ fails.
\end{theorem}

The first author came to this from a different angle, looking for ways to derive the strongest possible failure of these square principles from the assumption of certain forcing axioms. Forcing axioms such as Martin's Maximum ($\mathsf{MM}$) or the subcomplete forcing axiom ($\mathsf{SCFA}$) imply reflection principles of the form $\refl(\omega_1,S^\lambda_\omega)$ for sufficiently large regular $\lambda$ (in the case of $\mathsf{MM}$, $\lambda>\omega_1$ is enough, while in the case of $\mathsf{SCFA}$, $\lambda>2^\omega$ is needed). So the above theorem yields only the failure of $\square(\lambda,{<}\omega)$ as a consequence of these types of stationary reflection principles.

Looking to find improvements of Theorem~\ref{hlh_thm_213} that would
remove the hypothesis that $S \subseteq S^\lambda_{\geq \kappa}$, Hayut and the second
author introduced the following definition.

\begin{definition}[\cite{Hayut-LambieHanson:SimultaneousReflectionAndSquare}]
\label{def:Full}
  A $\square(\lambda, {<}\kappa)$-sequence $\langle \calC_\alpha \st \alpha <
  \lambda, ~ \alpha \text{ limit}\rangle$ is \emph{full} if for unboundedly many
  $\alpha < \lambda$, there is a club of $\beta < \lambda$ such that
  $\alpha$ is a limit point of some $C \in \calC_\beta$.
\end{definition}

They then proved that the requirement that $S \subseteq S^\lambda_{\geq \kappa}$
can be removed if one only wants to preclude the existence of \emph{full}
$\square(\lambda, {<}\kappa)$-sequences.

\begin{theorem}[{\cite[Thm.~2.18]{Hayut-LambieHanson:SimultaneousReflectionAndSquare}}]
\label{thm:SimulReflectionImpliesNoFullSquare}
Suppose that $\kappa<\lambda$ are uncountable cardinals, where $\lambda$ is regular, and suppose that $\refl({<}\kappa,S)$ holds, for some stationary $S\subseteq\lambda$. Then there is no full $\square(\lambda,{<}\kappa)$-sequence.
\end{theorem}

So this theorem shows that there is no full $\square(\lambda,\omega_1)$-sequence if $\refl(\omega_1,S^\lambda_\omega)$ holds and $\lambda>\omega_2$, for example. The following theorem serves to preclude the existence of a sequence that's not full:

\begin{theorem}[{\cite[Thm.~2.20]{Hayut-LambieHanson:SimultaneousReflectionAndSquare}}]
  \label{hlh_thm}
Suppose that $\kappa<\lambda$ are regular, uncountable cardinals. If $\refl(2,\lambda)$ holds, then there is no $\square(\lambda,{<}\kappa)$ sequence that is not full.
\end{theorem}

Notice that $\refl(1, \lambda)$, and hence $\refl(2, \lambda)$, fails if
$\lambda = \mu^+$ where $\mu$ is regular, as witnessed by $S^\lambda_\mu$.
Therefore, Theorem \ref{hlh_thm} is only nontrivial if $\lambda$ is either
weakly inaccessible or the successor of a singular cardinal.
In addition, since the stationary reflection principles derived from the abovementioned forcing axioms only yield stationary reflection for subsets of $S^\lambda_\omega$, Theorem \ref{hlh_thm} is not particularly relevant for
the study of their consequences. For these reasons, the first author was led to introduce the diagonal stationary reflection principle, which we formulate in the following definition, along with some natural variants.

\begin{definition}
\label{defn:DiagonalReflection}
Let $\lambda$ be a regular cardinal, let $S\subseteq\lambda$ be stationary, and let $\kappa<\lambda$. The \emph{diagonal stationary reflection principle} $\DSR({<}\kappa,S)$ asserts that whenever $\langle S_{\alpha,i}\st\alpha<\lambda, ~ i<j_\alpha\rangle$ is a matrix of stationary subsets of $S$, where $j_\alpha<\kappa$ for every $\alpha<\lambda$, then there are a $\gamma<\lambda$ of uncountable cofinality and a club $F\subseteq\gamma$ such that for every
$\alpha\in F$ and every $i<j_\alpha$, $S_{\alpha,i}$ reflects at $\gamma$. The version of the principle in which we only require $j_\alpha\le\kappa$ is denoted $\DSR(\kappa,S)$.

We will also be considering weakenings of this principle. Let $\uDSR({<}\kappa, S)$
and $\sDSR({<}\kappa, S)$ result from modifying the definition of $\DSR({<}\kappa, S)$
by requiring that the subset $F\subseteq\gamma$ be unbounded in $\gamma$ or
stationary in $\gamma$, respectively. As with $\refl$, if $S$ is a stationary subset
of a regular uncountable cardinal $\lambda$, let $\DSR^*({<}\kappa, S)$
denote the statement that $\DSR({<}\kappa, S)$ holds and continues to hold in
any extension by a $\lambda$-directed closed forcing notion of size $\lambda$
(and similarly for $\uDSR^*$ and $\sDSR^*$).
\end{definition}

\begin{remark}
  It is easily seen that each of the above principles, as well as each of the other
  stationary reflection principles considered in this paper, is equivalent to
  the apparent strengthening formed by requiring the existence of
  stationarily many such $\gamma$ at which reflection holds as opposed to just one.
\end{remark}

The original point of introducing these reflection principles was the fact that,
for all regular $\lambda \geq \omega_2$, the principle
$\DSR^*(\omega_1, S^\lambda_\omega)$ follows from either
$\mathsf{MM}$ or $\mathsf{SCFA}+\mathsf{CH}$, and the first author proved
that $\DSR(\omega_1, S^\lambda_\omega)$ implies the failure of $\square(\lambda, \omega_1)$ if $\lambda>\omega_2$ and the failure of $\square(\lambda,\omega)$ if $\lambda=\omega_2$:

\begin{theorem}[{\cite[Thm.~3.4]{Fuchs:DiagonalReflection}}] \label{fuchs_thm}
Let $\kappa < \lambda$ be cardinals, with $\lambda$ regular, and suppose that $\DSR({<}\kappa,S)$ holds for some stationary set $S\subseteq\lambda$. Then $\square(\lambda,{<}\kappa)$ fails.
\end{theorem}

We will reduce the hypothesis of the theorem from $\DSR({<}\kappa, S)$ to
$\sDSR({<}\kappa, S)$ in Theorem \ref{sdsr_square_incompatibility_theorem}.

If $\lambda$ is a regular cardinal, $S \subseteq \lambda$ is stationary, and
$\kappa < \lambda$, then it is clear from Definition \ref{defn:DiagonalReflection} that
\[
  \DSR({<}\kappa, S) \Rightarrow \sDSR({<}\kappa, S) \Rightarrow \uDSR({<}\kappa, S)
  \Rightarrow \refl({<}\kappa, S).
\]
Our obvious initial question, raised in light of Theorems
\ref{hlh_thm_213} and \ref{fuchs_thm}, was whether $\refl(\omega_1,S)$
implies $\DSR(\omega_1,S)$ and, more generally, the extent to which the arrows
in the above sequence of implications can be either reversed or strengthened.
The paper is organized as follows. In Section \ref{sec:Implications}, we
establish some implications and equivalences in ZFC. We will also prove
our strengthening of Theorem \ref{fuchs_thm}, showing that $\sDSR({<}\kappa, S)$
for some stationary $S \subseteq \lambda$ is enough to ensure the failure of
$\square(\lambda, {<}\kappa)$. Section \ref{sec:Separation} contains results
that separate various stationary reflection principles by showing that
certain implications do not hold in ZFC. A consequence of the results of
this section will be the fact that, in general, none of the arrows in the
above sequence of implications is reversible. In Section \ref{sec:DSRandSRP}, we extend a
result of Larson \cite[Theorem 4.6]{Larson:SeparatingSRP} by proving, among other things, that
the strong reflection principle, $\mathsf{SRP}$, does not imply
$\uDSR(1, S^\lambda_\omega)$ for any regular $\lambda > \omega_2$.
Then, in Section \ref{sec:Square}, we prove that our strengthening of
Theorem \ref{fuchs_thm} is sharp in the sense that the principles
$\sDSR({<}\kappa, S)$ and $\square(\lambda, \kappa)$ are compatible with one
another for infinite regular cardinals $\kappa < \lambda$ and a stationary $S \subseteq
\lambda$. Finally, in Section \ref{sec:Questions}, we list some
open questions that are raised by our work.

\section{Implications}
\label{sec:Implications}

We begin this section with our strengthening of Theorem \ref{fuchs_thm}. Here and in the following, for a set of ordinals $C$, we will use the notation $\lim(C)$ for the set of limit points of $C$ below the supremum of $C$.

\begin{theorem} \label{sdsr_square_incompatibility_theorem}
  Suppose that $1 < \kappa < \lambda$ are cardinals, with $\lambda$ regular, and
  suppose that there is a stationary $S \subseteq \lambda$ for which
  $\sDSR({<}\kappa, S)$ holds. Then $\square(\lambda, {<}\kappa)$ fails.
\end{theorem}

\begin{proof}
  Suppose for sake of contradiction that $\vec{\calC} = \langle \calC_\alpha
  \mid \alpha \in \lim(\lambda) \rangle$ is a $\square(\lambda,
  {<}\kappa)$-sequence. For each $\alpha \in \lim(\lambda)$, let
  $\calC_\alpha = \{C_{\alpha, i} \mid i < j_\alpha\}$, where $j_\alpha < \kappa$.
  For each limit $\alpha \in \lim(\lambda)$ and each $i < j_\alpha$, let
  \[
    S_{\alpha, i} = \{\beta \in S \setminus (\alpha + 1) \mid \text{for all }
    k < j_\beta, ~ C_{\beta, k} \cap \alpha \neq C_{\alpha, i}\}.
  \]
  \begin{claim}
    There is $\alpha_0 < \lambda$ such that $S_{\alpha, i}$ is stationary
    for all limit $\alpha$ with $\alpha_0 \leq \alpha < \lambda$ and all $i < j_\alpha$.
  \end{claim}

  \begin{proof}
    Otherwise, there would be an unbounded subset $A \subseteq
    \lim(\lambda)$ such that, for all $\alpha \in A$, there is $i_\alpha < j_\alpha$
    such that $S_{\alpha, i_\alpha}$ is nonstationary. For each $\alpha \in A$,
    let $D_\alpha$ be a club in $\lambda$ such that $D_\alpha \cap S_{\alpha,
    i_\alpha} = \emptyset$. Define an ordering $<_T$ on $A$ as follows.
    For all $\alpha, \beta \in A$, set $\alpha <_T \beta$ if and only if
    $\alpha < \beta$ and $C_{\beta, i_\beta} \cap \alpha = C_{\alpha, i_\alpha}$.
    It is easily verified that $T = (A, <_T)$ is a tree.

    We claim that $T$ has no antichains of size $\kappa$. To this end, fix a set
    $B \in [A]^\kappa$. We will find two elements of $B$ that are
    $<_T$-comparable. Fix a $\gamma \in S \cap \bigcap_{\alpha \in B} D_\alpha$ with $\gamma>\sup(B)$.
    Then, for all $\alpha \in B$, we have $\gamma \notin S_{\alpha, i_\alpha}$,
    so there is $k_\alpha < j_\gamma$ such that $C_{\gamma, k_\alpha} \cap \alpha
    = C_{\alpha, i_\alpha}$. Since $j_\gamma < \kappa$, we can find
    $\alpha < \beta$ in $B$ such that $k_\alpha = k_\beta$. But then
    $C_{\beta, i_\beta} \cap \alpha = C_{\gamma, k_\beta} \cap \alpha =
    C_{\alpha, i_\alpha}$, so $\alpha <_T \beta$.

    The tree $T$ therefore has height $\lambda$ and all of its levels have size
    less than $\kappa$. Since $\kappa < \lambda$, it then
    follows from a result of Kurepa \cite{Kurepa}
    that $T$ has a cofinal branch. Let $A^*$ be such a cofinal branch. Then
    $A^*$ is cofinal in $\lambda$ and, for all $\alpha < \beta$ in $A^*$,
    we have $C_{\alpha, i_\alpha} = C_{\beta, i_\beta} \cap \alpha$. Therefore,
    $\bigcup_{\alpha \in A^*} C_{\alpha, i_\alpha}$ is a thread through
    $\vec{\calC}$, contradicting the assumption that $\vec{\calC}$ is a
    $\square(\lambda, {<}\kappa)$-sequence.
  \end{proof}

  Fix $\alpha_0 < \lambda$ as in the claim. By $\sDSR({<}\kappa, S)$, we can
  find a $\gamma \in S^\lambda_{>\omega} \setminus (\alpha_0 + 1)$ and a
  stationary $F \subseteq \gamma \setminus \alpha_0$ such that, for all
  $\alpha \in F$ and all $i < j_\alpha$, $S_{\alpha, i}$ reflects at
  $\gamma$. Fix an arbitrary $C \in \calC_\gamma$, and find an
  $\alpha \in \lim(C) \cap F$. Since $\alpha \in \lim(C)$, it follows that
  there is $i < j_\alpha$ such that $C \cap \alpha = C_{\alpha, i}$.
  Since $\alpha \in F$, we know that $S_{\alpha, i} \cap \gamma$ is stationary
  in $\gamma$, so we can find $\beta \in \lim(C) \cap S_{\alpha, i}$.
  Then $\beta > \alpha$ and, since $\beta \in \lim(C)$, there is $k < j_\beta$
  such that $C_{\beta, k} = C \cap \beta$. But then $C_{\beta, k} \cap \alpha
  = C \cap \alpha = C_{\alpha, i}$, contradicting the fact that $\beta \in
  S_{\alpha, i}$ and finishing the proof of the theorem.
\end{proof}

The diagonal stationary reflection principles we are dealing with here are closely related to Paul Larson's principle $\OSR_{\omega_2}$ from \cite{Larson:SeparatingSRP}. Here is a slight reformulation and generalization of the original principle.

\begin{definition}
\label{def:OSR}
Suppose that $\lambda \geq \omega_2$ is a regular cardinal and $S \subseteq \lambda$
is stationary. $\OSR(S)$ is the assertion that, for every sequence $\langle S_\alpha \st i<\lambda\rangle$ of stationary subsets of $S$, there is a $\delta\in S^{\lambda}_{>\omega}$ such that for all $\alpha<\delta$, $S_\alpha$ reflects at $\delta$.
\end{definition}

Larson wrote $\OSR_{\omega_2}$ for the principle $\OSR(S^{\omega_2}_\omega)$. The following lemma shows among other things that, for stationary $S \subseteq \lambda$,
the principle $\OSR(S)$ is equivalent to each of the principles $\DSR({<}\lambda,S)$, $\uDSR({<}\lambda,S)$ and $\sDSR({<}\lambda,S)$.
Thus, differences between the principles $\DSR({<}\kappa,S)$, $\uDSR({<}\kappa,S)$ and $\sDSR({<}\kappa,S)$ can only be observed when $\sup(S)>\kappa$.

\begin{lemma}
\label{equivalence_lemma}
  Suppose that $\lambda \geq \aleph_2$ is a regular cardinal and
  $S \subseteq \lambda$ is stationary. Then the following are equivalent.
  \begin{enumerate}
    \item $\uDSR({<}\lambda, S)$.
    \item $\DSR({<}\lambda, S)$.
    \item $\OSR(S)$.
    \item For every matrix $\langle S_{\alpha, i} \mid \alpha < \lambda, ~
    i < j_\alpha \rangle$ of stationary subsets of $S$, where $j_\alpha < \lambda$
    for all $\alpha < \lambda$, there is $\gamma \in
    S^\lambda_{>\omega}$ such that $S_{\alpha, i}$ reflects at
    $\gamma$ for all $\alpha < \gamma$ and $i < j_\alpha$.
  \end{enumerate}
\end{lemma}

\begin{proof}
  It is immediate that clause (4) implies clauses (1), (2), and (3), and that
  clause (2) implies clause (1). To see that clause (3) implies clause (4), suppose
  that $\OSR(S)$ holds and that we are given a matrix $\langle S_{\alpha, i}
  \mid \alpha < \lambda, ~ i < j_\alpha \rangle$ as in the statement of clause
  (4). Let $\pi : \lambda \rightarrow \bigcup_{\alpha < \lambda}\left(\{\alpha\} \times
  j_\alpha\right)$ be a bijection. Form a sequence $\langle T_\alpha \mid \alpha < \lambda \rangle$
  of stationary subsets of $S$ by letting $T_\alpha = S_{\pi(\alpha)}$ for all
  $\alpha < \lambda$. Let $C$ be the set of $\gamma < \lambda$ such that
  $\pi``\gamma = \bigcup_{\alpha < \gamma} (\{\alpha\} \times j_\alpha)$. Then
  $C$ is a club in $\lambda$. Using $\OSR(S)$, we can find $\gamma \in
  C \cap S^\lambda_{>\omega}$ such that $T_\alpha$ reflects at
  $\gamma$ for all $\alpha < \gamma$. By our definition of $T_\alpha$ and
  our choice of $\gamma$, it follows that $S_{\beta, i}$ reflects at
  $\gamma$ for all $\beta < \gamma$ and $i < j_\beta$, so
  $\gamma$ witnesses this instance of clause (4).

  It remains to argue that clause (1) implies
  clause (3). To this end, suppose that $\uDSR({<}\lambda, S)$ holds, and let
  $\langle S_\alpha \mid \alpha < \lambda \rangle$ be a sequence of stationary
  subsets of $S$. Define a matrix
  $\langle T_{\alpha, i} \mid \alpha < \lambda, ~ i < \alpha \rangle$ of stationary
  subsets of $S$ by letting $T_{\alpha, i} = S_i$ for all
  $\alpha < \lambda$ and $i < \alpha$. By $\uDSR({<}\lambda, S)$, we can find an
  ordinal $\gamma \in S^\lambda_{>\omega}$ and an unbounded subset
  $F \subseteq \gamma$ such that, for all $\alpha \in F$ and all $i < \alpha$,
  we have that $T_{\alpha, i}$ reflects at $\gamma$.

  We claim that $S_\beta$ reflects at $\gamma$ for all
  $\beta < \gamma$. To see this, fix $\beta < \gamma$, and let
  $\alpha = \min(F \setminus (\beta + 1))$. Then $T_{\alpha, \beta} = S_\beta$. But, since
  $\alpha \in F$, we know that $T_{\alpha, \beta}$ reflects at
  $\gamma$. Therefore, $\gamma$ witnesses this instance of $\OSR(S)$.
\end{proof}

We will now explore the relationships between diagonal and simultaneous reflection. Let $\lambda$ be a regular uncountable cardinal. As noted in the introduction,
$\uDSR({<}\kappa,\lambda)$ implies $\refl({<}\kappa,\lambda)$, because given any sequence $\langle S_i\st i<\bar{\kappa} \rangle$,  where $\bar{\kappa}<\kappa$, one can consider the matrix defined by setting $S_{\alpha,i}=S_i$ for all $\alpha<\lambda$ and $i<\bar{\kappa}$. It is a more interesting question whether diagonal stationary reflection implies any amount of simultaneous reflection when this is not explicitly built into the principle at hand. The following lemma provides one instance in which this is the case.

\begin{lemma}
\label{dsr_implies_simultaneous}
Let $\lambda$ be a regular uncountable cardinal, and let $S\subseteq\lambda$ be stationary. Then $\DSR(1,S)$ implies $\refl(\omega,S)$.
\end{lemma}

\begin{proof}
  Assume that $\DSR(1, S)$ holds, and
let $\langle T_n \st n<\omega\rangle$ be a sequence of
stationary subsets of $S$. We will find an ordinal $\delta \in S^\lambda_{>\omega}$
such that, for all $n < \omega$, $T_n$ reflects at $\delta$.
By shrinking the sets if necessary, we may assume that $\langle T_n \mid n < \omega \rangle$ is a
sequence of pairwise disjoint sets. For each $n<\omega$, let
$T_n=\dot{\bigcup}_{l<\omega}T_{n,l}$ be a partition of $T_n$ into stationary sets. Define a function $s:\omega\times\omega\rightarrow\omega\times\omega$ by setting
\[s(n,l)=
\left\{
\begin{array}{l@{\qquad}l}
\langle 0,l+1 \rangle & \text{if $n\ge l$,}\\
\langle n+1,l \rangle & \text{if $n<l$.}
\end{array}
\right.
\]
Notice that, denoting the $p$-fold application of $s$ by $s^p$ (for $p<\omega$),
$s$ is defined in such a way that the following holds.
\begin{claim}
\label{claim:EventuallyEverythingIsReached}
For any $\langle{n,l}\rangle\in\omega\times\omega$ and any $m<\omega$, there is a $p\in\omega$ such that the first component of $s^p(n,l)$ is $m$.
\end{claim}
Now define a function $f:\lambda\rightarrow\omega\times\omega$ by
\[f(\alpha)=
\left\{
\begin{array}{l@{\qquad}l}
s(n,l) & \text{if $\alpha\in T_{n,l}$},\\
\langle 0,0\rangle & \text{if $\alpha\notin\bigcup_{n,l<\omega}T_{n,l}$}.
\end{array}
\right.\]
Note that $\DSR(1,S)$ is equivalent to the assertion that, given any sequence $\langle S_\alpha\st\alpha<\lambda\rangle$ of stationary subsets of $S$, there are $\gamma$, $F$ such that
\begin{enumerate}
\item[$(*)$] $\gamma\in S^\lambda_{{>}\omega}$, $F\subseteq\gamma$ is club, and for all $\alpha\in F$, $S_\alpha$ reflects at $\gamma$.
\end{enumerate}
Define
\[ S_\alpha=T_{f(\alpha)}\]
for $\alpha<\lambda$, and let $\gamma$ and $F$ satisfy $(*)$ with respect to
$\langle S_\alpha \mid \alpha < \lambda \rangle$.
\begin{claim}
\label{claim:FindingBeta}
Let $\alpha\in F$. Then there is a $\beta\in F$ with $f(\beta)=s(f(\alpha))$.
\end{claim}

\begin{proof}
Since $\alpha\in F$, $S_\alpha\cap\gamma$ is stationary in $\gamma$, by $(*)$. By definition, $S_\alpha=T_{f(\alpha)}$, so $T_{f(\alpha)}\cap\gamma$ is stationary
in $\gamma$. Since $F$ is club in $\gamma$, there is a $\beta\in T_{f(\alpha)}\cap F$. By definition, then, $f(\beta)=s(f(\alpha))$, so $\beta$ is as wished.
\end{proof}

Now let $m<\omega$. We will show that $T_m$ reflects at $\gamma$. To see this, pick $\alpha\in F$. Using Claim \ref{claim:EventuallyEverythingIsReached}, let $p\in\omega$ be such that the first component of $s^p(f(\alpha))$ is $m$, say $s^p(f(\alpha))=\langle{m,l}\rangle$.
By applying Claim \ref{claim:FindingBeta} $p$ times, we see that there is a $\beta\in F$ such that $f(\beta)=s^p(f(\alpha))$. By $(*)$, $S_\beta\cap\gamma$ is stationary in $\gamma$. But $S_\beta=T_{m,l}\subseteq T_m$, so $T_m$ reflects at $\gamma$, as wished.
\end{proof}

The previous lemma is optimal, in the sense that $\DSR$ cannot be replaced with $\sDSR$, by Theorem \ref{stationary_simultaneous_thm}, and in the sense that $\DSR(1,S)/\refl(\omega,S)$ cannot be replaced by $\DSR(\mu,S)/\refl(\mu^+,S)$ for any uncountable $\mu$,
by Theorem \ref{DSR_simultaneous_thm}.
We doubt that other instances of this phenomenon are possible, but at present have no model
in which, for example, $\DSR(1,S^\lambda_\omega)$ holds but
$\refl(\omega_1,S^\lambda_\omega)$ fails.

\section{Separations}
\label{sec:Separation}

In this section, we prove results separating a number of stationary reflection
principles. Our standard strategy to show that a reflection principle $P$ does
not imply a reflection principle $P^*$ will be as follows.
We start in a model $V$ in which an indestructible version of $P$, or
some strengthening of $P$, holds and then
produce a forcing extension $V_1$ in which a counterexample to $P^*$ is added.
This model $V_1$ will be our desired model. To show that $P$ holds in $V_1$,
we construct a further forcing extension $V_2$ such that the forcing leading
from $V$ to $V_2$ preserves (the strengthening of) $P$. In the last step, we
show that we can transfer $P$ down from $V_2$ to $V_1$.

For simplicity and concreteness, we will primarily be considering reflection
principles of the form
$\OSR(S)$, $\DSR({<\kappa}, S)$, etc.\ in which $S = S^\lambda_\omega$ for some
regular cardinal $\lambda \geq \omega_2$. Our reasons for doing this are twofold.
First, these are the reflection principles that are implied by forcing axioms
such as $\mathsf{MM}$ and $\mathsf{SCFA}$, and therefore we feel they are of
the most interest. Second, doing so will simplify matters at various points
due to the fact that the stationarity of subsets of $S^\lambda_\omega$ is always
preserved by countably closed forcing. It is not in general true that the
stationarity of
subsets of $S^\lambda_\kappa$ for uncountable $\kappa$ is preserved by
$\kappa^+$-closed forcing; to arrange for this, one would have to consider issues
of approachability that we would for the most part rather avoid here.
Nonetheless, with a bit more
care and attention to approachability, the reader can adapt our proofs to apply
to situations in which $S$ is of the form $S^\lambda_{\leq \kappa}$ for certain
uncountable $\kappa$.

\subsection{Indestructible reflection principles} \label{indestructible_subsection}

In this subsection, we indicate how to arrange for the indestructible reflection
principles that will appear as hypotheses in later theorems. We first deal with
the reflection principles at inaccessible cardinals. First note that, if
$\lambda$ is weakly compact, then $\OSR(\lambda)$ holds. By standard arguments (see,
e.g., \cite[Example 16.2]{cummings_iterated_forcing}), if $\lambda$ is a weakly
compact cardinal, then there is a forcing extension in which $\lambda$ remains
weakly compact and its weak compactness is indestructible under $\lambda$-directed
closed forcing of size $\lambda$. In particular, in this forcing extension,
$\lambda$ is inaccessible and $\OSR^*(\lambda)$ holds (and hence, by
Lemma \ref{equivalence_lemma}, $\DSR^*({<}\lambda, \lambda)$ holds as well).

We now consider indestructible reflection principles at successors of regular
cardinals.

\begin{theorem} \label{thm:weakly_compact_osr}
  Suppose that $\mu$ is a regular uncountable cardinal and $\lambda > \mu$ is
  a weakly compact cardinal. Then there is a $\mu$-closed, $\lambda$-cc forcing
  extension in which $\lambda = \mu^+$ and $\OSR^*(S^{\mu^+}_{<\mu})$ holds.
\end{theorem}

\begin{proof}
  This follows from a straightforward modification of the proof of
  \cite[Theorem 3.22]{Hayut-LambieHanson:SimultaneousReflectionAndSquare}.
  A key component of the proof is the fact that the stationarity of
  subsets of $S^{\mu^+}_{<\mu}$ is preserved by $\mu$-closed forcing,
  which holds because, for regular $\mu$, the set $S^{\mu^+}_{<\mu}$
  is in the approachability ideal $I[\mu^+]$. For this and other facts
  on the approachability ideal, we direct the reader to
  \cite[\S 3]{eisworth}.
\end{proof}

\begin{remark}
  Note that $\OSR^*(S^{\mu^+}_{{<}\mu})$ is equivalent to each of
  $\DSR^*(\mu,S^{\mu^+}_{{<}\mu})$ and $\uDSR^*(\mu,S^{\mu^+}_{{<}\mu})$ by
  Lemma \ref{equivalence_lemma}. Also, by Theorem \ref{thm:weakly_compact_osr}
  and a result of Magidor \cite{Magidor:ReflectingStationarySets}, the existence
  of a regular uncountable cardinal $\mu$ for which $\OSR^*(S^{\mu^+}_{{<}\mu})$
  holds is equiconsistent over ZFC with the existence of a weakly compact
  cardinal.
\end{remark}

We will now turn to the consistency of very indestructible versions of diagonal reflection
that will include indestructible diagonal reflection at successors of singular cardinals.
We will use the concept of generically supercompact cardinals, due to Cummings and Foreman.

\begin{definition}
\label{def:IndestGenSC}
A cardinal $\kappa$ is \emph{generically supercompact}%
\footnote{Our definition of generic supercompactness differs slightly from what is called generically supercompact by Foreman \cite[Definition 11.2]{Foreman2009:GenericEmbeddings}. The version of indestructible generic supercompactness of Foreman \cite[Definition 11.4]{Foreman2009:GenericEmbeddings} implies ours. It should be pointed out that clause \ref{item:UnneededItem} in our definition follows from \ref{item:CriticalPointKappa}-\ref{item:PointwiseImageInM}, since $\lambda$ is assumed to be regular in $V$. We kept this clause to highlight the similarity to Foreman's versions of these concepts.}%
if $\kappa$ is the successor of a regular cardinal $\mu$ and, for every regular cardinal $\lambda>\kappa$, there is a $\mu$-closed forcing notion $\tilde{\mathbb{R}}$ such that whenever $H$ is $\tilde{\mathbb{R}}$-generic, there is, in $V[H]$, an elementary embedding
\[j:V\prec M\subseteq V[H]\]
with:
\begin{enumerate}[label=(\arabic*)]
  \item
  \label{item:CriticalPointKappa}
  $\crit(j)=\kappa$,
  \item $j(\kappa)>\lambda$,
  \item
  \label{item:PointwiseImageInM}
  $j``\lambda\in M$,
  \item
  \label{item:UnneededItem}
  $\sup(j``\lambda)<j(\lambda)$,
  \item $M\models\cf(\lambda)=\mu$.
\end{enumerate}
A cardinal $\kappa$ is \emph{indestructibly} generically supercompact if whenever $\mathbb{R}$ is a $\kappa$-directed closed notion of forcing and $G$ is $\mathbb{R}$-generic, then $\kappa$ is generically supercompact in $V[G]$.
\end{definition}

It is well-known by \cite{Laver:Indestructibility} that if $\kappa$ is supercompact, then there is a forcing extension in which $\kappa$ is indestructibly supercompact, meaning that $\kappa$ is supercompact and remains so in any further forcing extension obtained by $\kappa$-directed closed forcing. If $\kappa$ is indestructibly supercompact and $\mu<\kappa$ is a regular uncountable cardinal, then in the forcing extension obtained by collapsing $\kappa$ to become the cardinal successor of $\mu$, $\kappa$ is indestructibly generically supercompact in the sense of \cite[Definition 11.4]{Foreman2009:GenericEmbeddings}, which implies our version of indestructible generic supercompactness, as stated in the footnote to Definition \ref{def:IndestGenSC}; see \cite[Remark after Def.~11.4]{Foreman2009:GenericEmbeddings}.

\begin{lemma}
\label{lem:DSRfromGenSC}
Suppose that $\mu$ is an uncountable regular cardinal and $\kappa=\mu^+$ is generically supercompact. Then for every regular cardinal $\lambda>\kappa$, the principle $\DSR({<}\kappa, S^\lambda_\omega)$ holds.
\end{lemma}

\begin{proof}
To show that $\DSR({<}\kappa,S^\lambda_{\omega})$ holds, note that since $\kappa=\mu^+$, the principle is equivalent to $\DSR(\mu,S^\lambda_\omega)$. Thus, let a matrix $\langle S_{\alpha,i}\st\alpha<\lambda, ~ i<\mu\rangle$ of stationary subsets of $S^\lambda_\omega$ be given.

By the generic supercompactness of $\kappa$, let $H$ be generic for a $\mu$-closed forcing notion $\tilde{\mathbb{R}}$ such that, in $V[H]$, there is an elementary embedding $j:V\prec M$ satisfying the clauses listed in Definition \ref{def:IndestGenSC}.

Note that $\lambda$ has cofinality $\mu$ in $V[H]$, since this is true in $M$,
by clause (5) of Definition \ref{def:IndestGenSC}, and since $\mu$ is still regular
in $V[H]$, by the $\mu$-closure of $\tilde{\mathbb{R}}$. Thus, $\nu=\sup j``\lambda$
has cofinality $\mu$ as well. Temporarily fixing $\alpha<\lambda$ and $i<\mu$,
let us verify that in $V[H]$, $j``S_{\alpha,i}$ is stationary in $\nu$. First, it
follows that $S_{\alpha,i}$ is still stationary in $V[H]$, because
$S_{\alpha,i}\subseteq S^\lambda_\omega$ and $\tilde{\mathbb{R}}$ is at least
$\omega_1$-closed -- it is well-known that countably closed forcing preserves
stationary subsets of $S^\lambda_\omega$. Now, to see that $j``S_{\alpha,i}$ is
stationary in $\nu$, let $C\subseteq\nu$ be club, with $C\in V[H]$. We have
that $j``\lambda$ is closed under limits of cofinality less
than $\mu$ in $V[H]$, since $V$ and $V[H]$ have the same sequences of ordinals of
length less than $\mu$. Thus, $\bar{C}=j^{-1}``C$ is closed under limits of
cofinality less than $\mu$. Letting $\bar{C}'$ be the closure of $\bar{C}$ in
$\lambda$, $\bar{C}'$ is club in $\lambda$, and hence, as $S_{\alpha,i}$ is
stationary in $V[H]$ and $\bar{C}'\in V[H]$, there is a $\beta\in\bar{C}'\cap
S_{\alpha,i}\subseteq S^\lambda_\omega$, so $\cf(\beta)<\mu$. But this means
that $\beta\in\bar{C}\cap S_{\alpha,i}$, and hence $j(\beta)\in C\cap j``S_{\alpha,i}$.

Since $j``S_{\alpha,i}\subseteq j(S_{\alpha,i})\cap\nu$, this means that $j(S_{\alpha,i})$
reflects at $\nu$ in $V[H]$, and hence in $M$. Let $\vec{T}=j(\vec{S})=\langle
T_{\alpha,i}\st\alpha<j(\lambda), ~ i<\mu\rangle$. Since in $M$, the cofinality
of $\nu$ is $\mu$ and $j``\lambda$ is closed under limits of cofinality less
than $\mu$, there is in $M$ a club $F\subseteq\nu$ of order type $\mu$ with
$F\subseteq j``\lambda$ -- recall that $j``\lambda\in M$, by clause \ref{item:PointwiseImageInM} of Definition
\ref{def:IndestGenSC}. Thus, $\nu$ and $F$ witness that
\begin{eqnarray*}
M\models\exists\nu'<j(\lambda)\exists F'&&(F'\subseteq\nu'\ \text{is club, } \cf(\nu')=\mu\ \\
&& \text{and}\ \forall\alpha\in F'\forall i<\mu\quad T_{\alpha,i}\ \text{reflects at}\ \nu').
\end{eqnarray*}%
This uses the fact that $F\subseteq j``\lambda$, so that if $\alpha\in F$, then $\alpha=j(\bar{\alpha})$ for some $\bar{\alpha}<\lambda$, and hence $T_{\alpha,i}=j(S_{\bar{\alpha},i})$ reflects at $\nu$.
The same statement is then true in $V$ about $\vec{S}$, by the elementarity of $j$. But this means that, in $V$ there are $\nu'<\lambda$ and $F'$ such that $\cf(\nu')=\mu$, $F'\subseteq\nu'$ is club, $\otp(F')=\mu$, and, for all $\alpha\in F'$ and all $i<\mu$, $S_{\alpha,i}$ reflects at $\nu'$, as wished.
\end{proof}

\begin{corollary}
\label{cor:IndestDSRfromIndestGenSC}
Suppose that $\mu$ is an uncountable regular cardinal and $\kappa=\mu^+$ is indestructibly generically supercompact. Then if $G$ is generic for a $\kappa$-directed closed notion of forcing and if $\lambda>\kappa$ is regular in $V[G]$, then the principle $\DSR({<}\kappa, S^\lambda_{\omega})$ holds in $V[G]$.
\end{corollary}

\begin{remark} \label{reflection_remark}
  The only way in which we used that assumption that we
  were only working with stationary subsets of $S^\lambda_\omega$ in the
  above proofs is the fact that the stationarity of subsets of
  $S^\lambda_\omega$ is preserved by $\mu$-closed forcing. This property is
  also satisfied by all stationary subsets of $S^\lambda_{<\mu}$ that
  lie in the approachability ideal $I[\lambda]$. If $\nu$ is a regular
  cardinal and $\nu^+ < \lambda$, then there is a stationary set
  $\Sigma \subseteq S^{\lambda}_\nu$ with $\Sigma \in I[\lambda]$ (cf.\
  \cite[\S 3.3]{eisworth}). Therefore, if $\nu < \mu$ is a fixed regular
  cardinal, then we can find a stationary $\Sigma
  \subseteq S^\lambda_\nu$ such that, in the conclusions of Lemma
  \ref{lem:DSRfromGenSC} and Corollary \ref{cor:IndestDSRfromIndestGenSC},
  $\DSR({<}\kappa, S^\lambda_\omega)$
  can be strengthened to $\DSR({<}\kappa, S^\lambda_\omega \cup \Sigma)$.
  This will be relevant below, in particular in Theorem
  \ref{thm:sDSRDoesNotImplyDSR}.
\end{remark}

\subsection{Forcing preliminaries}

In this subsection, we introduce some basic forcing notions and facts that will
be used in our separation results. An essential tool will be the following canonical
forcing notion to destroy the stationarity of a set by ``shooting a club through
its complement.''

\begin{definition}
\label{def:KillingStationarySets}
Let $\lambda$ be an uncountable regular cardinal, and let $S\subseteq\lambda$ be unbounded. The forcing notion $\bb{T}_S$ consists of all closed, bounded subsets $t$ of $\lambda$ such that $t \cap S = \emptyset$. The ordering is defined by setting, for $t_0, t_1 \in \bb{T}_S$, $t_1 \leq_{\bb{T}_S} t_0$ if
$t_1$ \emph{end-extends} $t_0$, that is, $t_1\cap\sup\{\xi+1\st\xi\in t_0\}=t_0$.
\end{definition}

In many natural cases, the forcing $\bb{T}_S$ is $\lambda$-distributive,
i.e., forcing with it does not add any new sequences of ordinals of length
less than $\lambda$. For example, if
$\lambda\setminus S$ is \emph{fat}, meaning that for every club $C\subseteq\lambda$ and every
$\alpha<\lambda$, there is a closed set $D\subseteq C\cap(\lambda\setminus S)$ of order type
$\alpha$, and if $\lambda$ is inaccessible, or $\lambda=\mu^+$, where $\mu$ is regular and
$\mu^{{<}\mu}=\mu$, then $\bb{T}_S$ is $\lambda$-distributive and adds a club subset of
$\lambda$ that is disjoint from $S$; see \cite[Theorem 1]{AbrahamShelah:ForcingClubs}. In
particular, this is the case if $\lambda=\omega_1$ and $\omega_1\setminus S$ is stationary.
So in these situations, the assumptions of the following lemma are satisfied.
In all of our applications, $\bb{T}_S$ will be $\lambda$-distributive.

\begin{lemma}
\label{lem:BasicsOnT_S}
Let $\lambda$ be an uncountable regular cardinal, and let $S\subseteq\lambda$ be such that both $S$ and $\lambda\setminus S$ are unbounded in $\lambda$.
\begin{enumerate}[label=(\arabic*)]
  \item
  \label{item:AddingAClub}
  Let $G$ be $\bb{T}_S$-generic over $V$. Then
  $\bigcup G$ is a club subset of $\lambda$ that is disjoint from $S$.
  \item
  \label{item:TFAE}
  Suppose that $\bb{T}_S$ preserves the fact that $\lambda$ has uncountable cofinality (in our applications, $\bb{T}_S$ will be $\lambda$-distributive, so this will certainly be the case).
  Then, for a set $T\subseteq\lambda$ such that $T\setminus S\subseteq S^\lambda_\omega$, the following are equivalent:
  \begin{enumerate}[label=(\alph*)]
    \item
    \label{item:T_SpreservesTAlways}
    $\forces_{\bb{T}_S}$ ``$\check{T}$ is stationary in $\check{\lambda}$.''
    \item
    \label{item:T_SpreservesTSometimes}
    There is a $t\in\bb{T}_S$ such that $t\forces_{\bb{T}_S}$ ``$\check{T}$ is stationary in $\check{\lambda}$.''
    \item
    \label{item:TminusSstationary}
    $T\setminus S$ is stationary.
  \end{enumerate}
\end{enumerate}
\end{lemma}

\begin{proof}
Claim \ref{item:AddingAClub} is immediate, since the fact that $\lambda \setminus
S$ is unbounded implies that, for every $\alpha<\lambda$, the set
$\{p\in\bb{T}_S\st\max(p)>\alpha\}$ is dense in $\bb{T}_S$.

Let us now prove claim \ref{item:TFAE}.

The implication \ref{item:T_SpreservesTAlways}$\implies$\ref{item:T_SpreservesTSometimes} is obvious.

For the implication \ref{item:T_SpreservesTSometimes}$\implies$\ref{item:TminusSstationary}, let $t\in\bb{T}_S$ be as in \ref{item:T_SpreservesTSometimes}, suppose $T$ is stationary in $V[G]$, where $G$ is generic for $\bb{T}_S$ with $t\in G$, and assume that $T\setminus S$ is not stationary in $V$. Let $C$ be a club subset of $\lambda$ disjoint from $T\setminus S$. Let $D=\bigcup G$, so $D$ is a club subset of $\lambda$ disjoint from $S$, by \ref{item:AddingAClub}. Since $\bb{T}_S$ preserves the fact that $\lambda$ has uncountable cofinality, $C\cap D$ is a club subset of $\lambda$. But $T=(T\cap S)\dot{\cup}(T\setminus S)$,
$(T\cap S)\cap D=\emptyset$ and $(T\setminus S)\cap C=\emptyset$, so $T\cap(C\cap D)=\emptyset$, contradicting the fact that $T$ is stationary in $V[G]$.

For the implication \ref{item:TminusSstationary}$\implies$\ref{item:T_SpreservesTAlways}, assume
that $T\setminus S\subseteq S^\lambda_\omega$ is stationary, but suppose that there are a
$\dot{C}\in V^{\bb{T}_S}$ and a condition $t\in\bb{T}$ forcing that $\dot{C}$ is a club subset
of $\check{\lambda}$ disjoint from $\check{T}$. Since $T\setminus S$ is stationary, we can find
an elementary submodel $M\prec H_\theta$, where $\theta$ is some sufficiently large regular
cardinal, such that the cardinality of $M$ is less than $\lambda$, $\delta=M\cap\lambda\in
T\setminus S$, and $M$ contains all relevant objects, including $t$, $S$, $T$ and $\dot{C}$.
Since $\delta\in (T\setminus S)\subseteq S^\lambda_\omega$, we can fix a strictly
increasing sequence of ordinals $\langle\xi_n\st n<\omega\rangle$ cofinal in $\delta$.
We can then simultaneously define sequences $\langle t_n\st n<\omega\rangle$ and
$\langle\alpha_n\st n<\omega\rangle$ by recursion on $n$ so that
the following conditions are satisfied for all $n<\omega$:
\begin{enumerate}
  \item $\xi_n<\alpha_n<\alpha_{n+1}<\delta$,
  \item $t_n\in\bb{T}_S\cap M$,
  \item $t_{n+1}\le t_n$, $t_0\le t$, $\max(t_n)>\xi_n$,
  \item $t_n\forces_{\bb{T}_S}\check{\alpha}_n\in\dot{C}$.
\end{enumerate}
For the construction, in the case $n=0$, working in $M$, since $t$ forces that $\dot{C}$ is
unbounded in $\delta$, there are a $t_0\le t$ and an $\alpha_0>\xi_0$ such that $t_0$ forces
that $\check{\alpha}_0\in\dot{C}$. $t_0$ may be chosen so that $\max(t_0)>\xi_0$, since
$\delta\setminus S$ is unbounded in $\delta$, by elementarity. Similarly, if $t_n$, $\alpha_n$
have already been defined, then, again working in $M$, there are a $t_{n+1}\le t_n$ with
$\max(t_{n+1})>\xi_{n+1}$ and an $\alpha_{n+1}>\xi_{n+1}$ such that $t_{n+1}$ forces that
$\check{\alpha}_{n+1}\in\dot{C}$. This completes the construction, and we can set
$t^*=(\bigcup_{n<\omega}t_n)\cup\{\delta\}$. Since $\delta\notin S$, it follows that
$t^*\in\bb{T}_S$, and since $t^*\le t_n$, it follows that $t^*$ forces that
$\alpha_n\in\dot{C}$, for every $n<\omega$. Since for all $n<\omega$, $\delta>\alpha_n>\xi_n$,
it follows that $\sup_{n<\omega}\alpha_n=\delta$. Thus, since $t^*$ forces that $\dot{C}$ is
closed below $\check{\lambda}$, we have that $t^*$ forces
$\check{\delta}\in\dot{C}\cap\check{T}$, a contradiction.
\end{proof}

We briefly recall the concept of strategic closure.

\begin{definition}
\label{def:StrategicallyClosed}
Let $\bb{P}$ be a notion of forcing, and let $\beta$ be an ordinal. In the game
$\Game_\beta(\bb{P})$, two players collaborate to play a $\leq_{\bb{P}}$-descending sequence
$\langle p_\alpha\st\alpha<\beta\rangle$ as follows. Player II plays at even stages of the
game (including limit stages) and in round $0$ must play $p_0=1_{\bb{P}}$.
If, during the course of play, a limit ordinal $\alpha < \beta$ is reached such
that $\langle p_\xi \mid \xi < \alpha \rangle$ has no lower bound in $\bb{P}$
(and hence Player II cannot continue playing), then Player I wins. Otherwise,
Player II wins. $\bb{P}$ is \emph{$\beta$-strategically closed} if Player
II has a winning strategy in the game $\Game_\beta(\bb{P})$.
\end{definition}

Observe that if $\bb{P}$ is $\lambda$-strategically closed for some regular cardinal $\lambda$, then $\bb{P}$ is $\lambda$-distributive.

\subsection{Separation results}

We are now ready to turn to our separation results. The first result
shows that even the strongest simultaneous reflection principles do not imply
unbounded diagonal reflection.

\begin{theorem}
\label{thm:SimultaneousReflDoesNotImplyUnbounded}
  Suppose that %
  $\lambda=\lambda^{{<}\lambda} \geq \aleph_2$ is a regular cardinal, $1 < \kappa \leq \lambda$,
  and $\refl^*({<}\kappa, S^\lambda_\omega)$ holds. Then
  there is a cofinality-preserving forcing extension in which
  \begin{enumerate}
    \item $\refl({<}\kappa, S^\lambda_\omega)$ holds;
    \item $\uDSR(1, S^\lambda_\omega)$ fails.
  \end{enumerate}
\end{theorem}

\begin{proof}
  Let $S = S^\lambda_\omega$, and let $\bb{P}$ be the forcing poset whose conditions
  are all functions of the form $p: \gamma^p \times \gamma^p \rightarrow 2$ such that
  \begin{itemize}
    \item $\gamma^p < \lambda$;
    \item for each $\alpha < \gamma^p$, letting $S^p_\alpha = \{\eta < \gamma^p
    \mid p(\alpha, \eta) = 1\}$, we have
    \begin{itemize}
      \item $S^p_\alpha \subseteq S \setminus (\alpha + 1)$;
      \item for all $\beta \in S^{\gamma^p + 1}_{>\omega}$, the set
      \[
        \{\alpha < \beta \mid S^p_\alpha \cap \beta \text{ is stationary in } \beta\}
      \]
      is bounded below $\beta$;
    \end{itemize}
    \item for all $\alpha < \beta < \gamma^p$, we have $S^p_\alpha \cap S^p_\beta = \emptyset$.
  \end{itemize}
  For all $p,q \in \bb{P}$, we say that $q \leq_{\bb{P}} p$ if $\gamma^q \geq
  \gamma^p$ and $q \restriction (\gamma^p \times \gamma^p) = p$.

  \begin{claim} \label{strat_closed_claim}
    $\bb{P}$ is $\lambda$-strategically closed.
  \end{claim}

  \begin{proof}
    We describe a winning strategy for Player II in $\Game_\lambda(\bb{P})$.
    In a run of $\Game_\lambda(\bb{P})$, the players will produce a
    $\leq_{\bb{P}}$-descending sequence $\vec{p} = \langle p_\eta \mid \eta < \lambda \rangle$.
    For each $\eta \in \lim(\lambda)$, we will let $\delta_\eta =
    \sup\{\gamma^{p_\xi} \mid \xi < \eta\}$ and then let $D = \{\delta_\eta \mid
    \eta \in \lim(\lambda)\}$. $D$ will be a club in $\lambda$.
    Also, for each $\alpha < \lambda$, we will let $S_\alpha =
    \bigcup_{\eta < \lambda} S^{p_\eta}_\alpha$.
    Player II will play in a way that ensures that, for all $\alpha < \lambda$,
    $S_\alpha \cap D = \emptyset$. Since $S_\alpha \cap (\alpha + 1) = \emptyset$ for
    all $\alpha < \lambda$, this amounts to ensuring that, for all $\eta \in
    \lim(\lambda)$ and all $\alpha < \delta_\eta$, we have
    $\delta_\eta \notin S_\alpha$.

    To start, Player II must play the empty condition as $p_0$. Next, suppose
    that $\eta < \lambda$ is an odd ordinal and $\langle p_\xi \mid \xi \leq \eta \rangle$
    has been played. Player II then plays any condition $p_{\eta + 1}$ with
    $\gamma^{p_{\eta + 1}} > \gamma^{p_\eta}$. Finally, suppose that
    $\eta < \lambda$ is a limit ordinal and $\langle p_\xi \mid \xi < \eta \rangle$
    has been played. Recall that $\delta_\eta = \sup\{\gamma^{p_\xi} \mid
    \xi < \eta\}$. Player II will play a condition $p_\eta$ with
    $\gamma^{p_\eta} = \delta_\eta + 1$, defined by letting
    $p_\eta \restriction (\delta_\eta \times \delta_\eta) = \bigcup_{\xi < \eta} p_\xi$
    and $p_\eta(\alpha, \eta) = 0$ for all $(\alpha, \eta) \in (\gamma^{p_\eta}
    \times \gamma^{p_\eta}) \setminus (\delta_\eta \times \delta_\eta)$.
    To see that $p_\eta$ is indeed a condition, it remains to show that, if
    $\cf(\delta_\eta) > \omega$, then the set
    \[
      \{\alpha < \delta_\eta \mid S^{p_\eta}_\alpha \cap \delta_\eta
      \text{ is stationary in } \delta_\eta\}
    \]
    is bounded below $\delta_\eta$. But, in fact, this set is empty: if
    $\cf(\delta_\eta) > \omega$, then $D \cap \delta_\eta = \{\delta_\xi \mid
    \xi \in \lim(\eta)\}$ is a club in $\delta_\eta$, and Player II's previous
    plays at limit stages have ensured that, for every $\alpha <\delta_\eta$,
    we have $S^{p_\eta}_\alpha \cap D \cap \delta_\eta = \emptyset$, so
    $S^{p_\eta}_\alpha \cap \delta_\eta$ is nonstationary in $\delta_\eta$.
    It follows that $p_\eta$ is indeed a condition in $\bb{P}$
    and that this describes a winning strategy for Player II in $\Game_\lambda(\bb{P})$.
  \end{proof}

  Note that it is implicit in this proof that $\bb{P}$ is countably closed.
  It is now straightforward to see that, for every $\alpha < \lambda$, the set
  $E_\alpha = \{p \in \bb{P} \mid \alpha < \gamma^p\}$ is a dense open subset
  of $\bb{P}$. Suppose that $G$ is a $\bb{P}$-generic filter over $V$. It follows
  that $f = \bigcup G$ is a function from $\lambda \times \lambda$ to $2$. For
  $\alpha < \lambda$, let $S_\alpha = \{\eta < \lambda \mid f(\alpha, \eta) = 1\}$.

  $V[G]$ is our desired model. Since $\bb{P}$ is $\lambda$-strategically closed
  and of size $\lambda=\lambda^{{<}\lambda}$, we know that $V[G]$ is a cofinality-preserving
  forcing extension of $V$. The following claim will allow us to conclude that
  the sequence $\langle S_\alpha \mid \alpha < \lambda \rangle$ witnesses
  the failure of $\uDSR(1, S^\lambda_\omega)$.

  \begin{claim} \label{stat_claim_i}
    In $V[G]$, for all $\alpha < \lambda$, $S_\alpha$ is a stationary subset of $S$.
  \end{claim}

  \begin{proof}
    Fix $\alpha < \lambda$. The fact that $S_\alpha$ is a subset of $S$ follows immediately from the
    definition of $\bb{P}$. To see that $S_\alpha$ is stationary in $V[G]$,
    work in $V$, let $\dot{S}_\alpha$ be the canonical $\bb{P}$-name for $S_\alpha$, let
    $p \in \bb{P}$, and let $\dot{C}$ be a $\bb{P}$-name such that
    $p \Vdash_{\bb{P}}``\dot{C} \text{ is a club in }\check{\lambda}"$.
    We will find a condition $q \leq p$ and an ordinal $\eta \in S$ such that
    $q \Vdash_{\bb{P}}``\check{\eta} \in \dot{S}_\alpha \cap \dot{C}"$.

    We will define a decreasing sequence of conditions $\langle p_k \mid k \leq \omega \rangle$
    and an increasing sequence of ordinals $\langle \eta_k \mid k < \omega \rangle$.
    To begin, let $p_0$ be any condition extending $p$ such that
    $\gamma^{p_0} > \alpha$. If $k < \omega$ and $p_k$ has been specified, let
    $p_{k+1}$ be any extension of $p_k$ such that there is an ordinal
    $\eta_k$ for which
    \begin{itemize}
      \item $\gamma^{p_k} < \eta_k < \gamma^{p_{k+1}}$ and
      \item $p_{k+1} \Vdash_{\bb{P}}``\check{\eta}_k \in \dot{C}"$.
    \end{itemize}
    This is straightforward given the discussion in the paragraph preceding
    the claim and the fact that $p$ forces $\dot{C}$ to be unbounded in $\lambda$.

    After $\langle p_k \mid k < \omega \rangle$ has been defined, let
    $\eta = \sup\{\gamma^{p_k} \mid k < \omega\} = \sup\{\eta_k \mid k < \omega\}$,
    and define a condition $p_\omega$ with $\gamma^{p_\omega} = \eta + 1$
    by letting
    \begin{itemize}
      \item $p_\omega \restriction (\eta \times \eta) = \bigcup_{k < \omega} p_k$;
      \item $p_\omega(\alpha, \eta) = 1$;
      \item $p_\omega(\beta, \xi) = 0$ for all $(\beta, \xi) \in (\gamma^{p_\omega} \times \gamma^{p_\omega})
        \setminus ((\eta \times \eta) \cup \{(\alpha, \eta)\})$.
    \end{itemize}
    It is clear that $p_\omega$ is a lower bound for $\langle p_k \mid k < \omega \rangle$.
    Also, since $p_\omega(\alpha, \eta) = 1$, we have
    $p_\omega \Vdash_{\bb{P}} ``\check{\eta} \in \dot{S}_\alpha"$. Moreover, for all
    $k < \omega$, $p_\omega \Vdash_{\bb{P}} ``\check{\eta}_k \in \dot{C}"$, so,
    since $p_\omega$ extends $p$ and $p$ forces $\dot{C}$ to be a club,
    we also have $p_\omega \Vdash_{\bb{P}}``\check{\eta} \in \dot{C}"$, as desired.
  \end{proof}

  The fact that $\langle S_\alpha \mid \alpha < \lambda \rangle$ witnesses
  the failure of $\uDSR(1, S^\lambda_\omega)$ in $V[G]$ now follows immediately
  from the previous claim and the definition of $\bb{P}$. It remains to verify
  that $\refl({<}\kappa, S)$ holds. To do this, we need to
  introduce some auxiliary forcing notions. In $V[G]$, for all
  $\beta < \lambda$, let $S_{{\ge}\beta} =
  \bigcup_{\beta \leq \gamma < \lambda} S_\gamma$, and let $\bb{T}_\beta=\bb{T}_{S_{{\geq}\beta}}$
  be the forcing to shoot a club through $\lambda$ disjoint from $S_{\geq \beta}$, introduced in Definition \ref{def:KillingStationarySets}.

  \begin{claim}
  \label{dense_closed_claim}
    In $V$, for all $\beta < \lambda$, $\bb{P} * \dot{\bb{T}}_\beta$ has a
    dense $\lambda$-directed closed subset of size $\lambda$.
  \end{claim}

  \begin{proof}
    Fix $\beta < \lambda$. For notational simplicity, let us assume that the name
    $\dot{\bb{T}}_\beta$ is closed under equivalent names, in the sense that it has
    the following property: whenever $p\in\bb{P}$, $\dot{t}$ and $\dot{u}$ are such
    that $(p,\dot{t})\in\bb{P}*\dot{\bb{T}}_\beta$, $\dot{u}\in V^{\bb{P}}\cap H_\lambda$
    and $p\forces_{\bb{P}}\dot{t}=\dot{u}$, it follows that $(p,\dot{u})\in
    \bb{P}*\dot{\bb{T}}_\beta$. (In the future, we will assume without
    comment that names for forcing posets are closed under equivalent names.)
    Let $\bb{U}_\beta$ be the set of conditions in
    $\bb{P}*\dot{\bb{T}}_\beta$ of the form $(p, \check{t})$ with
    $\gamma^p=\max(t)+1>\beta$. We will show that $\bb{U}_\beta$ has the desired properties.

    First, since $\lambda^{<\lambda} = \lambda$, it follows that $\bb{U}_\beta$ has size
    $\lambda$.

    To see that $\bb{U}_\beta$ is dense, fix
    $(p_0, \dot{t}_0) \in \bb{P} * \dot{\bb{T}}_\beta$. Since $\bb{P}$ is
    $\lambda$-strategically closed and hence does not add new bounded subsets
    of $\lambda$, by strengthening $p_0$ if necessary, we can assume that $p_0$
    decides the value of $\dot{t}_0$, and hence we can assume that $\dot{t}_0$
    is of the form $\check{t}_0$, where $t_0 \in V$. We can also assume that
    $\gamma^{p_0} > \max\{\max(t_0), \beta\}$. Now define a condition
    $p \leq p_0$ with $\gamma^p = \gamma^{p_0} + 1$ and
    $p(\alpha, \gamma^{p_0}) = 0$ for all $\alpha < \gamma^p$. Next, let
    $t = t_0 \cup \{\gamma^{p_0}\}$. Then $(p,\check{t})\in\bb{U}_\beta$ extends $(p_0,\dot{t}_0)$, showing that $\bb{U}_\beta$ is dense.

    We next show that $\bb{U}_\beta$ is $\lambda$-directed closed. Note that
    $\bb{U}_\beta$ is \emph{tree-like}, i.e., if $u,v,w \in \bb{U}_\beta$
    and $w$ extends both $u$ and $v$, then $u$ and $v$ are
    $\leq_{\bb{U}_\beta}$-comparable. It thus suffices to show that $\bb{U}_\beta$
    is $\lambda$-closed. To this end, let $\eta < \lambda$ be a limit ordinal,
    and suppose that $\langle (p_\xi, \check{t}_\xi) \mid \xi < \eta \rangle$ is a
    strictly decreasing sequence of conditions from $\bb{U}_\beta$.

    Let $\delta = \sup\{\gamma^{p_\xi} \mid \xi < \eta\}$.
    We begin by defining a condition $p \in \bb{P}$
    extending $\langle p_\xi \mid \xi < \eta \rangle$ with $\gamma^p =
    \delta + 1$. We do this by letting $p \restriction (\delta \times \delta)
    = \bigcup_{\xi < \eta} p_\xi$ and $p(\alpha, \epsilon) = 0$ for all
    $(\alpha, \epsilon) \in (\gamma^p \times \gamma^p) \setminus (\delta \times
    \delta)$. To verify that $p$ is a condition, it suffices to show that, if
    $\cf(\delta) > \omega$, then $\{\alpha < \delta \mid S^p_\alpha \cap \delta
    \text{ is stationary in } \delta\}$ is bounded below $\delta$. In fact,
    this set does not contain any ordinals greater than $\beta$, which
    will yield the desired conclusion since $\delta > \gamma^{p_0} > \beta$.
    To see this, note that, if $\cf(\delta) > \omega$, then $t' =
    \bigcup_{\xi < \eta} t_\xi$ is a club in $\delta$ and, by the definition
    of $\bb{T}_\beta$, it follows that, for all $\alpha$ with
    $\beta \leq \alpha < \delta$, we have $S^p_\alpha \cap t' = \emptyset$,
    so $S^p_\alpha \cap \delta$ is nonstationary in $\delta$.

    Finally, let $t = t' \cup \{\delta\}$. Then $(p,\check{t}) \in \bb{U}_\beta$ is a lower bound for the sequence given at the outset, thus showing that $\bb{U}_\beta$ is $\lambda$-directed closed, which completes the proof of the claim.
  \end{proof}

  Since $\bb{P}$ is $\lambda$-strategically closed in $V$ and, for all
  $\alpha<\lambda$, $\bb{P}*\bb{T}_{\alpha}$ has a $\lambda$-directed closed
  dense subset, it follows that $\bb{T}_\alpha$ is $\lambda$-distributive
  in $V[G]$. In particular,
  Lemma \ref{lem:BasicsOnT_S} applies, showing that the following are equivalent
  in $V[G]$, for a stationary subset $T$ of $S$:
  \begin{itemize}
    \item $\Vdash_{\bb{T}_\alpha} ``\check{T}\text{ is stationary in } \check{\lambda}"$.
    \item $T \setminus S_{\geq \alpha}$ is stationary.
  \end{itemize}
  Also note that, if $\alpha \leq \beta < \lambda$ and $T \setminus S_{\geq \alpha}$
  is stationary, then trivially $T \setminus S_{\geq \beta}$ is also stationary.
  With this in mind, the following claim will be instrumental in proving that
  $\refl({<}\kappa, S)$ holds in $V[G]$.

  \begin{claim} \label{capturing_claim}
    In $V[G]$, for every stationary $T \subseteq S$, there is $\alpha < \lambda$
    such that $T \setminus S_{\geq \alpha}$ is stationary.
  \end{claim}

  \begin{proof}
    Fix a stationary $T \subseteq S$, and suppose to the contrary that, for every
    $\alpha < \lambda$, $T \setminus S_{\geq \alpha}$ is nonstationary. Then,
    for each $\alpha < \lambda$, there is a club $C_\alpha$ in $\lambda$ such that
    $C_\alpha \cap T \subseteq S_{\geq \alpha}$. Let $C = \Delta_{\alpha < \lambda}
    C_\alpha$. Then $C$ is a club in $\lambda$, so we can fix some $\beta \in C
    \cap T$. Then $\beta \in T \cap \bigcap_{\alpha < \beta}S_{\geq \alpha} =
    T \cap S_{\geq \beta}$. But, for all $\gamma$ with $\beta \leq \gamma < \lambda$,
    we have $S_\gamma \cap (\gamma + 1) = \emptyset$, so $\beta \notin
    S_{\geq \beta}$. This is a contradiction.
  \end{proof}

  To see that $\refl({<}\kappa, S)$ holds in $V[G]$, fix $\mu < \kappa$ and a
  sequence $\langle T_\zeta \mid \zeta < \mu \rangle$ of stationary subsets of
  $S$. By Claim \ref{capturing_claim}, for each $\zeta < \mu$, there is
  $\alpha_\zeta < \lambda$ such that $T_\zeta \setminus S_{\geq \alpha_\zeta}$
  is stationary. Let $\beta = \sup\{\alpha_\zeta \mid \zeta < \mu\}$.
  Since $\mu < \kappa \leq \lambda$, we have $\beta < \lambda$, and
  $T_\zeta \setminus S_{\geq \beta}$ is stationary for all $\zeta < \mu$.
  Let $H$ be $\bb{T}_\beta$-generic over $V[G]$.

  By Claim \ref{dense_closed_claim}, in $V$, $\bb{P} * \dot{\bb{T}}_\beta$
  has a dense $\lambda$-directed closed subset of size $\lambda$. Since
  $\refl^*({<}\kappa, S)$ holds in $V$, this implies that $\refl({<}\kappa, S)$
  holds in $V[G * H]$. For all $\zeta < \mu$, since $T_\zeta \setminus S_{\geq \beta}$
  is stationary in $V[G]$, we know that $T_\zeta$ is stationary in $V[G*H]$.
  It follows that, in $V[G * H]$, there is $\delta \in S^\lambda_{> \omega}$
  such that $T_\zeta \cap \delta$ is stationary in $\delta$ for all $\zeta < \mu$.
  Since stationarity is downward absolute, the same holds in $V[G]$,
  so $\langle T_\zeta \mid \zeta < \mu \rangle$ reflects simultaneously in $V[G]$.
  Since this sequence was chosen arbitrarily, it follows that $\refl({<}\kappa, S)$
  holds in $V[G]$.
\end{proof}

We now show that even simultaneous unbounded diagonal reflection does not imply stationary
diagonal reflection. Note that, by Lemma \ref{equivalence_lemma}, the assumption
that $\kappa<\lambda$ in the following theorem is necessary.

\begin{theorem}
\label{unbounded_stationary_thm}
  Suppose that 
  $\lambda=\lambda^{{<}\lambda}\geq \aleph_2$ is a regular
  cardinal, $\kappa$ is a cardinal with $1 < \kappa < \lambda$, and
  $\uDSR^*({<}\kappa, S^\lambda_\omega)$ holds. Then there is a cofinality-preserving
  forcing extension in which
  \begin{enumerate}
    \item $\uDSR({<}\kappa, S^\lambda_\omega)$ holds;
    \item $\sDSR(1, S^\lambda_\omega)$ fails.
  \end{enumerate}
\end{theorem}

\begin{proof}
  Let $S = S^\lambda_\omega$, and let $\bb{P}$ be the forcing poset whose
  conditions are all functions of the form $p:\gamma^p \times \gamma^p
  \rightarrow 2$ such that
  \begin{itemize}
    \item $\gamma^p < \lambda$;
    \item for each $\alpha < \gamma^p$, letting $S^p_\alpha = \{\eta < \gamma^p
    \mid p(\alpha, \eta) = 1\}$, we have
    \begin{itemize}
      \item $S^p_\alpha \subseteq S \setminus (\alpha + 1)$;
      \item for all $\beta \in S^{\gamma^p+1}_{> \omega}$, the set
      \[
        \{\alpha < \beta \mid S^p_\alpha \cap \beta \text{ is stationary in } \beta\}
      \]
      is nonstationary in $\beta$;
    \end{itemize}
    \item for all $\alpha < \beta < \gamma^p$, we have $S^p_\alpha \cap S^p_\beta = \emptyset$.
  \end{itemize}
  For $p, q \in \bb{P}$, we say that $q \leq_{\bb{P}} p$ if and only if
  $\gamma^q \geq \gamma^p$ and $q \restriction (\gamma^p \times \gamma^p) = p$.

  Let $G$ be $\bb{P}$-generic over $V$ and, for $\alpha < \lambda$, let $S_\alpha
  = \bigcup_{p \in G} S^p_\alpha$. Arguments exactly as in the proofs of Claims
  \ref{strat_closed_claim} and \ref{stat_claim_i} yield the truth of the following
  statements:
  \begin{itemize}
    \item In $V$, $\bb{P}$ is $\lambda$-strategically closed.
    \item In $V[G]$, for every $\alpha < \lambda$, $S_\alpha$ is stationary in
    $\lambda$. Moreover, if we let $S_{-1} = S \setminus (\bigcup_{\alpha < \lambda} S_\alpha)$,
    then $S_{-1}$ is stationary as well.
  \end{itemize}

  In $V[G]$, let $\bb{T}$ be the forcing notion whose conditions are all closed,
  bounded subsets $t$ of $\lambda$ such that, for all $\alpha \in t$, we have
  $t \cap S_\alpha = \emptyset$.
  As before, these conditions are ordered by end-extension.
  Let $\dot{\bb{T}} \in V$ be a canonical $\bb{P}$-name for $\bb{T}$.

  \begin{claim} \label{claim_315}
    In $V$, $\bb{P} * \dot{\bb{T}}$ has a dense $\lambda$-directed closed subset
    of size $\lambda$.
  \end{claim}

  \begin{proof}
    Define $\bb{U}$ to consist of all conditions in $\bb{P}*\dot{\bb{T}}$ that are of the form $(p,\check{t})$ with the property that $\gamma^p = \max(t) + 1$.
    Since $\lambda^{<\lambda} = \lambda$, the cardinality of $\bb{U}$ is $\lambda$.

    To see that $\bb{U}$ is dense, fix $(p_0, \dot{t}_0)
    \in \bb{P} * \dot{\bb{T}}$. Since $\bb{P}$ is $\lambda$-strategically closed,
    we can assume, by extending $p_0$ if necessary, that $p_0$ decides $\dot{t}_0$,
    and hence we may assume that $\dot{t}_0$ is of the form $\check{t}_0$. We may also assume that $\gamma^{p_0} > \max(t_0)$. Now define a condition $p \leq_{\bb{P}} p_0$ by letting $\gamma^{p} = \gamma^{p_0} + 1$, $p\restriction(\gamma^{p_0}\times\gamma^{p_0})=p_0$, and
    $p(\alpha, \gamma^{p_0})=p(\gamma^{p_0},\alpha) = 0$ for all
    $\alpha\le\gamma^p$. Next, let $t = t_0 \cup \{\gamma^{p_0}\}$. Then $(p,\check{t})$ is a condition in $\bb{U}$ extending $(p_0,\dot{t}_0)$, showing that $\bb{U}$ is dense.

    We next show that $\bb{U}$ is $\lambda$-directed closed. Since $\bb{U}$ is
    tree-like, it suffices to show $\lambda$-closure. To this end,
    let $\eta < \lambda$ be a limit ordinal, and suppose that $\langle (p_\xi,
    \check{t}_\xi) \mid \xi < \eta \rangle$ is a strictly decreasing sequence of
    conditions from $\bb{U}$. Let
    $\delta = \sup\{\gamma^{p_\xi} \mid \xi < \eta\} = \sup\{\max(t_\xi) \mid
    \xi < \eta\}$.

    We first construct a condition $p \in \bb{P}$ such that $p$ is a lower bound
    for $\langle p_\xi \mid \xi < \eta \rangle$ and $\gamma^p = \delta + 1$. To
    do this, simply let $p \restriction (\delta \times \delta) = \bigcup_{\xi < \eta}
    p_\xi$ and $p(\alpha, \zeta) = 0$ for all $(\alpha, \zeta) \in (\gamma^p \times
    \gamma^p) \setminus (\delta \times \delta)$. The only nontrivial statement to
    check to verify that $p$ is indeed a condition is that, if
    $\cf(\delta) > \omega$, then the set
    \[
      X = \{\alpha < \delta \mid S^p_\alpha \cap \delta \text{ is stationary in } \delta\}
    \]
    is nonstationary in $\delta$. But notice that, if $\cf(\delta) > \omega$,
    then $t_\infty = \bigcup_{\xi < \eta} t_\xi$ is a club in $\delta$ and, for
    all $\alpha \in t_\infty$, we have $t_\infty \cap S^p_\alpha = \emptyset$.
    Therefore, $X$ is disjoint from $t_\infty$ and hence nonstationary in $\delta$. Letting $t = t_\infty \cup\{\delta\}$, it now follows that $(p,\check{t})$ is a condition in $\bb{U}$ that is a lower bound for the given sequence.
  \end{proof}

  $V[G]$ is our desired model. It is clear from what has been written that, in
  $V[G]$, $\langle S_\alpha \mid \alpha < \lambda \rangle$ is a witness to the
  failure of $\sDSR(1, S)$. To verify that $\uDSR({<}\kappa, S)$ holds, the
  following claim will be useful.

  \begin{claim} \label{stat_preserve_claim}
    In $V[G]$, suppose that $-1 \leq \epsilon < \lambda$, $R$ is a stationary
    subset of $S_{\epsilon}$, and $t_0 \in \bb{T}$ is a condition such that
    $\epsilon \notin t_0$ and $\max(t_0) > \epsilon$. Then $t_0 \Vdash_{\bb{T}}
    ``\check{R} \text{ is stationary in } \check{\lambda}"$.
  \end{claim}

  \begin{proof}
    Work in $V[G]$. Let $t \leq_{\bb{T}} t_0$ be arbitrary, and let $\dot{C}$ be a $\bb{T}$-name
    forced by $t$ to be a club in $\lambda$. It will suffice to find $r \leq_{\bb{T}}
    t$ and $\delta \in R$ such that $r \Vdash_{\bb{T}}``\check{\delta} \in \dot{C}"$.

    We will proceed much as in the proof of Lemma \ref{lem:BasicsOnT_S}. Let $\theta$ be a sufficiently large regular cardinal, let $\vartriangleleft$
    be a fixed well-ordering of $H_\theta$, and let $M$ be an elementary
    submodel of $(H_\theta, \in, \vartriangleleft)$ of size less than $\lambda$ such that
    \begin{itemize}
      \item $\{\bb{T}, t, \dot{C}, R\} \subseteq M$;
      \item $\delta := \sup(M \cap \lambda) \in R$.
    \end{itemize}
    Let $\langle \delta_n \mid n < \omega \rangle$ enumerate a cofinal subset of
    $M \cap \lambda$. Now, working at each step inside $M$, recursively
    construct a decreasing sequence of conditions
    $\langle r_n \mid n < \omega \rangle$ from $\bb{T} \cap M$ together with an
    increasing sequence $\langle \xi_n \mid n < \omega \rangle$ of ordinals from
    $M \cap \lambda$ such that
    \begin{itemize}
      \item $r_0 \leq_{\bb{T}} t$;
      \item for each $n < \omega$, we have $\delta_n < \min\{\max(r_n), \xi_n\}$;
      \item for each $n < \omega$, $r_n \Vdash_{\bb{T}}``\check{\xi}_n \in \dot{C}"$.
    \end{itemize}
    The construction is straightforward, using the fact that $t$ forces $\dot{C}$
    to be unbounded in $\lambda$. Notice that $\delta = \sup\{\max(r_n) \mid n < \omega\}
    = \sup\{\xi_n \mid n < \omega\}$. Also recall that $\delta \in R \subseteq S_\epsilon$.
    Since $\epsilon \notin t_0$ and $\max(t_0) > \epsilon$, it follows that
    $r := \{\delta\} \cup \bigcup_{n < \omega} r_n$ is a condition in $\bb{T}$
    that is a lower bound for $\langle r_n \mid n < \omega \rangle$ and hence
    extends $t$. For each $n < \omega$, $r \Vdash_{\bb{T}}``\check{\xi}_n \in \dot{C}"$,
    so, since $r$ forces $\dot{C}$ to be a club in $\lambda$, we see that
    $r \Vdash_{\bb{T}} ``\check{\delta} \in \dot{C} \cap \check{R}"$, as desired.
  \end{proof}

  Note also that, letting $S_{\geq \alpha} = \bigcup_{\alpha \leq \beta < \lambda}
  S_\beta$ for each $\alpha < \lambda$, the proof of Claim \ref{capturing_claim} applies
  in this case as well to yield the fact that, in $V[G]$, for every stationary
  $R \subseteq S$, there is $\alpha < \lambda$ such that $R \setminus S_{\geq \alpha}$
  is stationary. But for such an $\alpha$, we then have
  $R\setminus S_{{\ge}\alpha}\subseteq\bigcup_{-1\le\beta<\alpha}S_\beta$,
  and hence, by the completeness of the nonstationary ideal, there is in fact
  some $-1\le\beta<\alpha$ such that $R\cap S_\beta$ is stationary.


  To see that $\uDSR({<}\kappa, S)$ holds in $V[G]$, fix a matrix
  $\langle R_{\alpha, i} \mid \alpha < \lambda, ~ i < j_\alpha \rangle$
  of stationary subsets of $S$ such that, for each $\alpha < \lambda$, we have
  $j_\alpha < \kappa$. We will find $\gamma \in S^\lambda_{>\omega}$ and an
  unbounded $F \subseteq \gamma$ such that, for all $\alpha \in F$ and all
  $i < j_\alpha$, we have that $R_{\alpha, i}$ reflects at $\gamma$.

  For each $\alpha<\lambda$ and $i<j_\alpha$, let $\beta(\alpha, i)<\lambda$ be such
  that $R_{\alpha, i} \cap S_{\beta(\alpha, i)}$ is stationary.
  Set $B_\alpha = \{\beta(\alpha, i) \mid i < j_\alpha\}$, and let
  $\langle\ell^\alpha_\xi\st\xi<\theta_\alpha\rangle$ enumerate $B_\alpha$
  in increasing order. Fix a stationary set $\Sigma^* \subseteq \lambda$ on
  which the map $\alpha \mapsto (j_\alpha, \theta_\alpha)$ is constant,
  with value $(j, \theta)$.

  Let $\rho$ be the least ordinal below $\theta$ such that the set $\{ \ell^\alpha_\rho
  \mid \alpha \in \Sigma^* \}$ is unbounded in $\lambda$, if such an ordinal
  exists, and let $\rho = \theta$ otherwise. For all $\xi < \rho$, then, there
  is $\beta_\xi < \lambda$ such that $\ell^\alpha_\xi < \beta_\xi$ for all
  $\alpha \in \Sigma^*$. Let $\varepsilon = \sup_{\xi < \rho}
  \beta_\xi < \lambda$, and let $t = \{\varepsilon\} \in \bb{T}$.

  \begin{claim} \label{density_claim}
    For all $\zeta < \lambda$, the set
    \[
      D_\zeta := \left\{t^* \in \bb{T} \mid \exists \alpha \in \Sigma^* \setminus \zeta ~
      \left[\forall i < j \left(t^* \Vdash_{\bb{T}} ``\check{R}_{\alpha, i}
      \text{ is stationary in } \check{\lambda}"\right)\right]\right\}
    \]
    is dense in $\bb{T}$ below $t$.
  \end{claim}

  \begin{proof}
    Fix $\zeta < \lambda$ and $t' \leq_{\bb{T}} t$. Fix $\alpha \in \Sigma^*$ such
    that $\alpha \geq \zeta$ and, if $\rho < \theta$, then $\ell^\alpha_\rho >
    \max(t')$. Now find $\gamma > \max(t')$ with
    $\sup\{\ell^\alpha_\xi \mid \xi < \theta\} < \gamma < \lambda$ and
    $\gamma \in S_{-1}$, and let
    $t^* = t' \cup \{\gamma\}$. Then $t^*$ is a condition in $\bb{T}$ extending
    $t'$, and we claim that, for all $i < j$, $t^* \Vdash_{\bb{T}}``\check{R}_{\alpha, i}
    \text{ is stationary in }\check{\lambda}"$, and hence $t^* \in D_\zeta$.

    To see this, fix $i < j$, and let $\xi < \theta$ be such that $\beta(\alpha, i)
    = \ell^\alpha_\xi$. If $\xi < \rho$, then $\beta(\alpha, i) < \varepsilon =
    \min(t^*)$. Therefore, since $R_{\alpha, i} \cap S_{\beta(\alpha, i)}$ is stationary,
    Claim \ref{stat_preserve_claim} implies that $t^*$ preserves the stationarity
    of $R_{\alpha, i} \cap S_{\beta(\alpha, i)}$ and hence, \emph{a fortiori},
    of $R_{\alpha, i}$ itself. If, on the other hand, $\rho \leq \xi < \theta$,
    then we have
    \[
      \max(t') < \ell^\alpha_\rho \leq \ell^\alpha_\xi = \beta(\alpha, i)
      < \gamma = \min(t^* \setminus (\max(t') + 1)).
    \]
    In particular, we have $\beta(\alpha, i) \notin t^*$ and $\beta(\alpha, i)
    < \max(t^*)$, so Claim \ref{stat_preserve_claim} again implies that
    $t^*$ preserves the stationarity of $R_{\alpha, i}$.
  \end{proof}

  Let $H$ be $\bb{T}$-generic over $V[G]$ with $t\in H$. Claim \ref{density_claim}
  implies that, in $V[G*H]$, the set
  \[
  A := \{\alpha \in \Sigma^* \mid \text{for all } i < j, ~ R_{\alpha, i} \text{ is stationary in } \lambda\}
  \]
  is unbounded in $\lambda$. Let $f:\lambda\longrightarrow A$ be the monotone enumeration of $A$, and let $C$ be the set of closure points of $f$, that is, the set of $\eta<\lambda$ such that $f``\eta\subseteq\eta$. Note that $C$ is a club in $\lambda$. Note also that since, in $V$, $\uDSR^*({<}\kappa, S)$
  holds and $\bb{P} * \dot{\bb{T}}$ has a dense $\lambda$-directed closed subset
  of cardinality $\lambda$, it follows that $\uDSR({<}\kappa, S)$ holds in
  $V[G * H]$. Therefore, we can apply this principle to the matrix $\langle
  R'_{\xi,i}\st\xi<\lambda, ~ i<j\rangle$ defined by $R'_{\xi,i}=R_{f(\xi),i}\cap C$,
  yielding an ordinal $\gamma\in S^\lambda_{>\omega}$ and an unbounded set $F'\subseteq\gamma$ such that, for all $\xi \in F'$ and all $i < j$, $R'_{\xi, i}$ reflects at $\gamma$, i.e., $R_{f(\xi), i}\cap C$ reflects at $\gamma$.
  In particular, $\gamma\in C$, and hence $F = f``F'$ is an unbounded subset of $\gamma$. Moreover, since $V[G]$ and $V[G*H]$ have the same bounded
  subsets of $\lambda$, we also know that $F \in V[G]$ and, in $V[G]$, for all
  $\alpha \in F$ and all $i < j$, $R_{\alpha, i} \cap \gamma$ is stationary in
  $\gamma$. We have thus found $\gamma$ and $F$ as desired, so we have shown
  that $\uDSR({<}\kappa, S)$ holds in $V[G]$.
\end{proof}

We now show that stationary diagonal reflection implies no more simultaneous reflection than explicitly stated. Recall that this is not true of the full diagonal reflection principle, by Lemma \ref{dsr_implies_simultaneous}.

\begin{theorem}
\label{stationary_simultaneous_thm}
  Suppose that 
  $1 < \kappa < \lambda=\lambda^{{<}\lambda}$ are cardinals, $\lambda \geq \aleph_2$ is regular, and either $\kappa\ge\omega$ and $\DSR^*({<}\kappa, S^\lambda_\omega)$ holds, or $\kappa<\omega$ and $\DSR^*(\kappa,S^\lambda_\omega)$ holds. Then there is a cofinality-preserving forcing extension in which
  \begin{enumerate}
    \item $\sDSR({<}\kappa, S^\lambda_\omega)$ holds;
    \item $\refl(\kappa, S^\lambda_\omega)$ fails.
  \end{enumerate}
\end{theorem}

\begin{proof}
  Let $S = S^\lambda_\omega$, and let $\bb{P}$ be the forcing poset whose conditions
  are all functions of the form $p:\kappa \times \gamma^p \rightarrow 2$ such that
  \begin{itemize}
    \item $\gamma^p < \lambda$;
    \item for each $\ell < \kappa$, letting $S^p_\ell = \{\eta < \gamma^p \mid
    p(\ell, \eta) = 1\}$, we have
    \begin{itemize}
      \item $S^p_\ell \subseteq S$;
      \item for all $\beta \in S^{\gamma^p + 1}_{>\omega}$, there is $\ell < \kappa$
      such that $S^p_\ell \cap \beta$ is nonstationary in $\beta$;
    \end{itemize}
    \item for all $\ell < \ell^* < \kappa$, we have $S^p_\ell \cap S^p_{\ell^*}
    = \emptyset$.
  \end{itemize}
  For $p, q \in \bb{P}$, we say that $q \leq_{\bb{P}} p$ if and only if
  $\gamma^q \geq \gamma^p$ and $q \restriction (\kappa \times \gamma^p) = p$.

  Let $G$ be $\bb{P}$-generic over $V$ and, for $\ell < \kappa$, let
  $S_\ell = \bigcup_{p \in G} S^p_\ell$.
  By standard arguments analogous to those in the proofs of Claims \ref{strat_closed_claim}
  and \ref{stat_claim_i}, it follows that
  \begin{itemize}
    \item in $V$, $\bb{P}$ is $\lambda$-strategically closed;
    \item in $V[G]$, for every $\ell < \kappa$, $S_\ell$ is stationary in $\lambda$.
    Moreover, if we let $S_{-1} = S \setminus (\bigcup_{\ell < \kappa} S_\ell)$,
    then $S_{-1}$ is stationary as well.
  \end{itemize}

  In $V[G]$, for each $\ell < \kappa$, let $\bb{T}_\ell=\bb{T}_{S_\ell}$ be the forcing notion
  to shoot a club through $\lambda$ disjoint from $S_\ell$ introduced in Definition \ref{def:KillingStationarySets}.
  Let $\dot{\bb{T}}_\ell \in V$ be a canonical $\bb{P}$-name for $\bb{T}_\ell$. By arguments like those in
  the proof of Claim \ref{dense_closed_claim}, in $V$, $\bb{P} * \dot{\bb{T}}_\ell$
  has a dense $\lambda$-directed closed subset of cardinality $\lambda$.

  $V[G]$ is our desired model. It follows from the previous two paragraphs that
  the sequence $\langle S_\ell \mid \ell < \kappa \rangle$ witnesses the failure
  of $\refl(\kappa, S)$ in $V[G]$. To see that $\sDSR({<}\kappa, S)$ holds, let
  $\langle T_{\alpha, i} \mid \alpha < \lambda, ~ i < j_\alpha \rangle$ be a
  matrix of stationary subsets of $S$, with $j_\alpha < \kappa$ for all
  $\alpha < \lambda$. Notice that, for every stationary subset $T \subseteq S$,
  there is at most one $\ell < \kappa$ for which $T \setminus S_\ell$ is
  nonstationary. Therefore, for all $\alpha < \lambda$, there is some
  $\ell(\alpha) < \kappa$ such that, for all $i < j_\alpha$, we have that
  $T_{\alpha, i} \setminus S_{\ell(\alpha)}$ is stationary. In particular, by Lemma \ref{lem:BasicsOnT_S}, each $T_{\alpha, i}$ remains stationary after forcing over $V[G]$ with $\bb{T}_{\ell(\alpha)}$.

  Find a stationary $\Sigma \subseteq S_{-1}$ and a fixed $\ell < \kappa$ such
  that $\ell(\alpha) = \ell$ for all $\alpha \in \Sigma$. Let $H$ be
  $\bb{T}_\ell$-generic over $V[G]$. For all $\alpha \in \Sigma$ and
  $i < j_\alpha$, $T_{\alpha, i}$ remains stationary in $V[G * H]$. Moreover,
  since $\Sigma$ is a subset of $S_{-1}$ and hence disjoint from $S_\ell$,
  $\Sigma$ also remains stationary in $V[G*H]$. Since, in $V$, either $\kappa\ge\omega$ and $\DSR^*({<}\kappa, S)$ holds, or $\kappa<\omega$ and $\DSR^*(\kappa,S)$ holds, and since
  $\bb{P} * \dot{\bb{T}}_\ell$ has a dense $\lambda$-directed closed
  subset of cardinality $\lambda$, it follows that the corresponding diagonal reflection principle holds in $V[G * H]$.

  In $V[G*H]$, form a matrix $\langle T^*_{\alpha, i} \mid \alpha < \lambda, ~
  i < j^*_\alpha \rangle$ as follows. For $\alpha \in \Sigma$, let
  $j^*_\alpha = j_\alpha + 1$, let $T^*_{\alpha, i} = T_{\alpha, i}$ for
  $i < j_\alpha$, and let $T_{\alpha, j_\alpha} = \Sigma$. For
  $\alpha \in \lambda \setminus \Sigma$, let $j^*_\alpha = 1$ and
  $T_{\alpha, 0} = \Sigma$. We can apply our diagonal reflection principle to this
  matrix in $V[G * H]$, finding $\gamma \in
  S^\lambda_{> \omega}$ and a club $F^* \subseteq \gamma$ such that, for all
  $\alpha \in F^*$ and all $i < j^*_\alpha$, we have that $T^*_{\alpha, i}$
  reflects at $\gamma$. Since $\Sigma \in \{T^*_{\alpha, i}
  \mid i < j^*_\alpha\}$ for every $\alpha < \lambda$, it follows that
  $\Sigma$ reflects at $\gamma$, and hence $F := \Sigma \cap F^*$
  is stationary in $\gamma$ as well. Moreover, for each $\alpha \in F$ and
  each $i < j_\alpha$, we know that $T_{\alpha, i} = T^*_{\alpha, i}$
  reflects at $\gamma$. This is downward absolute to
  $V[G]$ and, since $V[G]$ and $V[G*H]$ have the same bounded subsets of
  $\lambda$, we know that $F$ is in $V[G]$. Therefore, in $V[G]$, $\gamma$
  and $F$ witness this instance of $\sDSR({<}\kappa, S)$.
\end{proof}

\begin{corollary}
\label{cor:sDSRfiniteDoesNotImplyDSR1}
  Suppose that 
  $\lambda=\lambda^{{<}\lambda} \geq \aleph_2$ is a regular
  cardinal, and $\DSR^*({<} \omega, S^\lambda_\omega)$ holds. Then there is a cofinality-preserving
  extension in which
  \begin{enumerate}
    \item $\sDSR({<}\omega, S^\lambda_\omega)$ holds;
    \item $\DSR(1, S^\lambda_\omega)$ fails.
  \end{enumerate}
\end{corollary}

\begin{proof}
  This follows immediately from Theorem \ref{stationary_simultaneous_thm} and
  Lemma \ref{dsr_implies_simultaneous}.
\end{proof}

We now prove a more general version of the above corollary, indicating that
simultaneous stationary diagonal reflection does not imply any amount of the
full diagonal stationary reflection principle. Recall that, by the discussion
in Subsection \ref{indestructible_subsection}, the hypothesis of the following
theorem can be arranged from a weakly compact cardinal if $\lambda$ is to be
inaccessible or the successor of a regular cardinal, or from a supercompact
cardinal if $\lambda$ is to be the successor of a singular cardinal.

\begin{theorem} \label{thm:sDSRDoesNotImplyDSR}
  Suppose that $\kappa < \lambda$ are infinite cardinals such that
  $\kappa^+ < \lambda$ and $\lambda^{{<}\lambda} = \lambda$. Suppose
  moreover that there is a stationary subset
  $\Sigma \subseteq S^\lambda_{\geq \kappa}$ such that
   $\DSR^*({<} \kappa, S^\lambda_\omega \cup \Sigma)$
  holds. Then there is a cofinality-preserving forcing extension in which
  \begin{enumerate}
    \item $\sDSR({<}\kappa, S^\lambda_\omega)$ holds;
    \item $\DSR(1, S^\lambda_\omega)$ fails.
  \end{enumerate}
\end{theorem}

\begin{proof}
  Let $S = S^\lambda_\omega$, and let $\bb{P}$ be the forcing poset whose
  conditions are all functions of the form $p:\gamma^p \times \gamma^p
  \rightarrow 2$ such that
  \begin{itemize}
    \item $\gamma^p < \lambda$;
    \item for each $\alpha < \gamma^p$, letting $S^p_\alpha = \{\eta
    < \gamma^p \mid p(\alpha, \eta) = 1\}$, we have
    \begin{itemize}
      \item $S^p_\alpha \subseteq S \setminus (\alpha + 1)$;
      \item for all $\beta \in S^{\gamma^p + 1}_{{>}\omega}$, the set
      \[
        \{\alpha < \beta \mid S^p_\alpha \cap \beta \text{ is nonstationary in
        } \beta\}
      \]
      is stationary in $\beta$;
    \end{itemize}
    \item for all $\alpha < \alpha^* < \gamma^p$, we have $S^p_\alpha
    \cap S^p_{\alpha^*} = \emptyset$.
  \end{itemize}
  Let $G$ be $\bb{P}$-generic over $V$ and, for $\alpha < \lambda$, let
  $S_\alpha = \bigcup_{p \in G} S^p_\alpha$. By arguments as in the
  proofs of Claims \ref{strat_closed_claim} and \ref{stat_claim_i},
  we have the following facts.
  \begin{itemize}
    \item In $V$, $\bb{P}$ is $\lambda$-strategically closed.
    \item In $V[G]$, for every $\alpha < \lambda$, $S_\alpha$ is stationary
    in $\lambda$. Moreover, if we let $S_{-1} = S \setminus (\bigcup_{\alpha
    < \lambda} S_\alpha)$, then $S_{-1}$ is stationary as well.
  \end{itemize}

  As a result, by the definition of conditions in $\bb{P}$,
  $\langle S_\alpha \mid \alpha < \lambda \rangle$ witnesses the failure
  of $\DSR(1, S^\lambda_\omega)$ in $V[G]$.

  In $V[G]$, we define a poset $\bb{T}$ designed to destroy the stationarity
  of $S_\alpha$ for a generic $\omega$-closed unbounded set of $\alpha$-s.
  Recall that a set $X$ of ordinals is $\omega$-closed if, for all
  ordinals $\alpha$, if $\cf(\alpha) = \omega$ and $\sup(X \cap \alpha)
  = \alpha$, then $\alpha \in X$.
  When defining similar forcings in previous arguments, it was sufficient to
  force a single club simultaneously destroying the stationarity of many
  stationary sets, but here it will pay dividends to more carefully add a distinct
  club disjoint from each $S_\alpha$. To be more precise, conditions in
  $\bb{T}$ are pairs $t = (c^t, d^t)$ satisfying the following requirements:
  \begin{itemize}
    \item $c^t$ is an $\omega$-closed bounded subset of $\lambda$;
    \item $d^t$ is a function with domain $c^t$ such that, for all
    $\alpha \in c^t$, $d^t(\alpha)$ is a closed bounded subset of
    $\lambda$ with $d^t(\alpha) \cap S_\alpha = \emptyset$.
  \end{itemize}
  If $t$ and $s$ are conditions in $\bb{T}$, then $s \leq_{\bb{T}} t$ if and only if
  \begin{itemize}
    \item $c^s$ end-extends $c^t$; and
    \item for all $\alpha \in c^t$, $d^s(\alpha)$ end-extends $d^t(\alpha)$.
  \end{itemize}
  In $V$, let $\dot{\bb{T}}$ be a $\bb{P}$-name for $\bb{T}$.
  \begin{claim}
    In $V$, $\bb{P}*\dot{\bb{T}}$ has a dense $\lambda$-directed closed subset of
    size $\lambda$.
  \end{claim}
  \begin{proof}
    Define $\bb{U}$ to be the set of $(p, \check{t}) \in
    \bb{P} * \dot{\bb{T}}$ such that, letting $t = (c^t, d^t)$, we have
    \begin{itemize}
      \item $c^t$ has a largest element;
      \item $\gamma^p = \max(c^t) + 1$;
      \item for all $\alpha \in c^t$, $\gamma^p = \max(d^t(\alpha)) + 1$.
    \end{itemize}
    The verification that $\bb{U}$ is a dense $\lambda$-directed closed
    subset of $\bb{P} * \dot{\bb{T}}$ of cardinality $\lambda$ is
    almost the same as that in the proof of Claim \ref{claim_315},
    here using the fact that, if $X$ is an $\omega$-closed set of ordinals,
    $\beta$ is an ordinal of uncountable cofinality, and $\sup(X \cap \beta)
    = \beta$, then $X \cap \beta$ is stationary in $\beta$. We therefore
    leave the rest of the proof to the reader.
  \end{proof}
  Work for now in $V[G]$, which will be our desired model.
  If $t \in \bb{T}$, then
  $\bb{T}/t$ denotes the set $\{s \in \bb{T} \mid s \leq t\}$, considered as a
  sub-poset of $\bb{T}$. Given a condition $t \in \bb{T}$ and an
  $\varepsilon < \lambda$, let $t \setminus \varepsilon$ denote the
  condition $(c^t \setminus \varepsilon, d^p \restriction (c^t
  \setminus \varepsilon))$. Furthermore, define a function
  $\pi_{t, \varepsilon} : \bb{T}/t \rightarrow \bb{T}/(t \setminus
  \varepsilon)$ by letting $\pi_{t, \varepsilon}(s) = s \setminus
  \varepsilon$ for all $s \in \bb{T}/t$.

  \begin{claim} \label{projection_claim}
    If $t \in \bb{T}$, $\varepsilon < \lambda$, and
    $c^t \setminus \varepsilon$ is non-empty, then $\pi_{t,
    \varepsilon}$ is a projection%
    \footnote{See, e.g., \cite[p.~335]{Abraham:ProperForcing}.}
    from $\bb{T}/t$ to
    $\bb{T}/(t \setminus \varepsilon)$.
  \end{claim}

  \begin{proof}
    It is clear that $\pi_{t,\varepsilon}$ is order-preserving and
    maps the top element of $\bb{T}/t$, namely $t$, to the top element of
    $\bb{T}/(t \setminus \varepsilon)$, namely $t \setminus \varepsilon$.
    It remains to show that,
    for all $s \in \bb{T}/t$ and all $r \leq (s \setminus \varepsilon)$,
    there is $s^* \leq s$ such that $(s^* \setminus \varepsilon) \leq r$.
    To this end, fix such $s$ and $r$. Since $c^t \setminus \varepsilon
    \neq \emptyset$ and $s \leq t$, it follows that for all $\alpha \in
    c^r \setminus c^s$, we have $\alpha\supseteq c^s$. Define a
    condition $s^* \in \bb{T}$ as follows:
    \begin{itemize}
      \item $c^{s^*} = c^r \cup c^s$;
      \item for all $\alpha \in c^r$, $d^{s^*}(\alpha) = d^r(\alpha)$;
      \item for all $\alpha \in c^s \setminus c^r$, $d^{s^*}(\alpha)
      = d^s(\alpha)$.
    \end{itemize}
    Note that $c^{s^*} \setminus \varepsilon = c^r$. It is now easily
    verified that $s^*$ is a condition in $\bb{T}$ extending $s$ and
    that $s^* \setminus \varepsilon$ extends $r$.
  \end{proof}

  \begin{claim} \label{claim_323}
    Suppose that, in $V[G]$, we have $t \in \bb{T}$, $\varepsilon < \lambda$,
    and a stationary $R \subseteq \lambda$ such that
    \begin{itemize}
      \item $c^t \setminus \varepsilon \neq \emptyset$; and
      \item $t \Vdash_{\bb{T}}``\check{R} \text{ is stationary in }\check{\lambda}"$.
    \end{itemize}
    Then $t \setminus \varepsilon \Vdash_{\bb{T}}``\check{R}
    \text{ is stationary in }\check{\lambda}"$.
  \end{claim}

  \begin{proof}
This is an instance of a more general fact: if $\bb{P}$ and $\bb{Q}$ are posets with a projection $\pi:\bb{Q}\longrightarrow\bb{P}$, $A$ is some set, $\varphi(x)$ is a statement that goes down to inner models, and for some $q\in\bb{Q}$, $q\forces_{\bb{Q}}\varphi(\check{A})$, then $\pi(q)\forces_{\bb{P}}\varphi(\check{A})$. If not, let $p\le\pi(q)$ be such that $p\forces_{\bb{P}}\neg\varphi(\check{A})$. Let $q^*\le q$ be such that $\pi(q^*)=p$. Let $I\ni q^*$ be $\bb{Q}$-generic. Then $\varphi(A)$ holds in $V[I]$, since $q^*\le q$. But $\pi``I$ generates a $V$-generic filter $\bar{I}$ for $\bb{P}$, and $\pi(q^*)=p\in\bar{I}$. So $\varphi(A)$ fails in $V[\bar{I}]$. But $V[\bar{I}]\subseteq V[I]$, so, since $\varphi(A)$ holds in $V[I]$, it must hold in $V[\bar{I}]$, a contradiction.
  \end{proof}

  \begin{claim} \label{claim_324}
    Suppose that $-1 \leq \epsilon < \lambda$, $R$ is a stationary
    subset of $S_\epsilon$, and $t \in \bb{T}$ is a condition
    such that $\epsilon \notin c^t$ and $\sup(c^t) > \epsilon$.
    Then $t \Vdash_{\bb{R}}``\check{R} \text{ is stationary in }\check{\lambda}
    "$.
  \end{claim}

  \begin{proof}
    The proof is almost exactly the same as that of Claim
    \ref{stat_preserve_claim}, so we leave it to the reader.
  \end{proof}

  It remains to verify that
  $\sDSR({<}\kappa, S)$ holds in $V[G]$. To this end, let
  $\langle T_{\alpha, i} \mid \alpha < \lambda, ~ i < j_\alpha \rangle$
  be a matrix of stationary subsets of $S$, with $j_\alpha < \kappa$
  for all $\alpha < \lambda$. As in the proof of Theorem
  \ref{unbounded_stationary_thm}, for each $\alpha < \lambda$
  and $i < j_\alpha$, we can find $\beta(\alpha, i)$ such that
  $-1 \leq \beta(\alpha, i) < \lambda$ and $T_{\alpha, i} \cap
  S_{\beta(\alpha, i)}$ is stationary.

  For each $\alpha < \lambda$, set
  \begin{itemize}
    \item $X_\alpha = \{\beta(\alpha, i) \mid i < j_\alpha\}$;
    \item $X_\alpha^- = X_\alpha \cap \alpha$;
    \item $X_\alpha^+ = X_\alpha \setminus \alpha$.
  \end{itemize}
  Note that $|X_\alpha| < \kappa$ for all $\alpha < \lambda$. Therefore,
  recalling that $\Sigma \subseteq S^\lambda_{\geq \kappa}$, we have
  $\sup(X_\alpha^-) < \alpha$ for all $\alpha \in \Sigma$.
  Then the function $\alpha \mapsto \sup(X_\alpha^-)$ is regressive on
  the stationary set $\Sigma$ and clearly continues to be
  in any forcing extension by $\bb{T}$. (Recall that $\bb{P} * \dot{\bb{T}}$
  has a dense $\lambda$-directed closed subset and hence preserves
  the stationarity of all stationary subsets of $\lambda$ that are in $V$;
  in particular, it preserves the stationarity of $\Sigma$).
  Therefore, we can find
  $\varepsilon < \lambda$ and $t_0 \in \bb{T}$ such that $t_0$
  forces that $R_{\varepsilon} := \{\alpha \in \Sigma
  \mid \sup(X_\alpha^-) =
  \varepsilon\}$ is stationary in $\lambda$. By extending $t_0$ if
  necessary, we may assume that $c^{t_0} \setminus (\varepsilon + 1)
  \neq \emptyset$. Let $t = t_0 \setminus (\varepsilon + 1)$. Then
  we have the following:
  \begin{itemize}
    \item By Claim \ref{claim_323}, $t \Vdash_\bb{T} ``\check{R}_{\varepsilon}
    \text{ is stationary}"$.
    \item By Claim \ref{claim_324}, for all $\alpha \in R_{\varepsilon}$
    and all $i < j_\alpha$ such that $\beta(\alpha, i) < \alpha$,
    $t \Vdash_{\bb{T}} ``\check{T}_{\alpha, i} \text{ is stationary}"$.
  \end{itemize}
  Let $\dot{R}^*$ be a $\bb{T}$-name for the set of $\alpha \in
  R_\varepsilon$ such that, for all $i < j_\alpha$, $T_{\alpha, i}$
  is stationary in $\lambda$ (after forcing with $\bb{T}$).

  \begin{claim}
    $t \Vdash_{\bb{T}}``\dot{R}^* \text{ is stationary}"$.
  \end{claim}

  \begin{proof}
    Fix $s_0 \leq t$ and a $\bb{T}$-name $\dot{E}$ for a club in $\lambda$.
    We must find $s \leq s_0$ and $\alpha \in R_\epsilon$ such that
    $s \Vdash_{\bb{T}}``\check{\alpha} \in \dot{E} \cap \dot{R}^*"$.
    Let $H$ be $\bb{T}$-generic over $V[G]$ with $s_0 \in H$, and let
    $E$ be the realization of $\dot{E}$. Let $c = \bigcup_{s \in H} c^s$ and
    define a function $d$ on $c$ by letting $d(\alpha) = \bigcup_{s \in H}
    d^s(\alpha)$. For each $\alpha \in E$, let
    $r_\alpha \in H$ be such that $r_\alpha \leq s_0$ and
    $r_\alpha \Vdash_{\bb{T}}``\check{\alpha}
    \in \dot{E}"$, and let $\xi_\alpha < \lambda$ be such that
    $c^{r_\alpha} \subseteq \xi_\alpha$. Then the set
    \[
      E^* = \{\alpha \in \lim(E) \mid \text{for all } \eta \in E \cap \alpha, ~
      \xi_\eta < \alpha\}
    \]
    is a club in $\lambda$. Since $s_0 \leq t$ and hence $t \in H$, we know
    that $R_{\varepsilon}$ is stationary in $V[G * H]$, so we can find
    $\alpha \in E^* \cap R_{\varepsilon}$. We can also find a
    $\gamma < \lambda$ large enough so that, for all $\eta < \alpha$ and all
    $\beta \in c^{r_\eta}$, we have $d^{r_\eta}(\beta) \subseteq \gamma$.
    Now define $s_1 = (c^{s_1}, d^{s_1})$ by letting $c^{s_1} = c \cap
    \alpha$ and, for all $\beta \in c$, letting $d^{s_1}(\beta) =
    d(\beta) \cap (\gamma + 1)$. Note that $s_1 \in \bb{T}$ and
    $s_1 \leq r_\eta$ for all $\eta < \alpha$. Therefore, $s_1$ forces that
    $\alpha$ is a limit point of $\dot{E}$.

    Now let $\delta < \lambda$ be large enough so that $\alpha < \delta$
    and $X_\alpha^+ \subseteq \delta$. Define a condition $s \leq s_1$ by letting
    $c^s = c^{s_1} \cup \{\delta\}$, $d^s \restriction c^{s_1} = d^{s_1}$,
    and $d^s(\delta) = \emptyset$. Notice that $c^s \cap \varepsilon =
    c^s \cap [\alpha, \delta) = \emptyset$, and therefore $X_\alpha \cap
    c^s = \emptyset$. Moreover, $\delta \in c^s$ and every element of
    $X_\alpha$ is less than $\delta$. Therefore, by Claim \ref{claim_324},
    \[
      s \Vdash_{\bb{T}}``\text{for all } i < j_\alpha, ~ \check{T}_{\alpha, i}
      \text{ is stationary}".
    \]
    Therefore, $s \Vdash_{\bb{T}}``\check{\alpha} \in \dot{E} \cap \dot{R}^*$,
    as desired.
  \end{proof}

  Now let $H$ be $\bb{T}$-generic over $V[G]$ with $t \in H$, and let $R^*$ be
  the interpretation of $\dot{R}^*$. Define a matrix $\langle T^*_{\alpha, i}
  \mid \alpha < \lambda, ~ i \leq j_\alpha \rangle$ of stationary subsets of
  $S^\lambda_\omega \cup \Sigma$ as follows. First, for all
  $\alpha < \lambda$, let $T^*_{\alpha, j_\alpha} = R^*$. Next, if
  $\alpha \in R^*$, then let $T^*_{\alpha, i} = T_{\alpha, i}$ for all
  $i < j_\alpha$. Finally, if $\alpha \in \lambda \setminus R^*$, then
  simply let $T^*_{\alpha, i} = S^\lambda_\omega$ for all $i < j_\alpha$.

  Since $\DSR^*({<}\kappa, S^\lambda_\omega \cup \Sigma)$ holds in
  $V$ and $\bb{P} * \dot{\bb{T}}$ has a dense $\lambda$-directed closed subset
  of size $\lambda$, $\DSR({<}\kappa, S^\lambda_\omega \cup \Sigma)$
  holds in $V[G*H]$. Therefore, we can find $\gamma \in S^\lambda_{>\omega}$
  and a club $F^*$ in $\gamma$ such that, for all $\alpha \in F^*$ and
  all $i \leq j_\alpha$, $T^*_{\alpha, i}$ reflects at $\gamma$.
  In particular, $R^*$ reflects at $\gamma$, so $F = F^* \cap R^*$ is stationary
  in $\gamma$. Moreover, for all $\alpha \in F$ and all $i < j_\alpha$,
  we know that $T^*_{\alpha, i} = T_{\alpha, i}$, so $T_{\alpha, i}$ reflects
  at $\gamma$. But stationarity is downward absolute and, since $\bb{T}$ is
  $\lambda$-distributive in $V[G]$, we have $F \in V[G]$, so, in $V[G]$,
  $\gamma$ and $F$ witness this instance of $\sDSR({<}\kappa, S^\lambda_\omega)$.
\end{proof}

\begin{remark}
  We do not know if the requirement that $\kappa^+ < \lambda$ is necessary
  in the statement of the above theorem. It was used in our proof to allow
  for the existence of a stationary $\Sigma \subseteq S^\lambda_{\geq \kappa}$
  that reflects.
\end{remark}

Recall that Lemma \ref{dsr_implies_simultaneous} showed that $\DSR(1,S)$ implies $\refl(\omega,S)$. The next theorem
provides a limitation to the extent to which this can be generalized. It also shows that Theorem \ref{sdr_square_thm}, stating that $\sDSR({<}\kappa,S)$ implies the failure of $\square(\lambda,{<}\kappa)$ (where $1<\kappa<\lambda$ and $S\subseteq\lambda$ is stationary), is in a sense optimal. The assumptions of the theorem hold, for example, if $\kappa$ is indestructibly generically supercompact.

\begin{theorem}
\label{DSR_simultaneous_thm}
  Suppose that $\omega_1<\kappa<\lambda$ are regular cardinals, and $\DSR({<}\kappa, S^\lambda_\omega)$ holds and continues to hold in any forcing extension obtained by a $\kappa$-directed closed forcing notion that preserves cofinalities up to $\lambda$. Then there is such a forcing extension in which
  \begin{enumerate}
    \item $\DSR({<}\kappa, S^\lambda_\omega)$ holds;
    \item $\refl(\kappa, S')$ fails for all stationary $S'\subseteq\lambda$.
  \end{enumerate}
  In fact, in this forcing extension, we have $\square^{\mathrm{ind}}(\lambda,\kappa)$,
  a strengthening of $\square(\lambda,\kappa)$, see \ref{def:IndexedSquare}.
\end{theorem}

\begin{proof}
Under the assumptions of the theorem, there is a forcing notion $\mathbb{S}$ that is $\kappa$-directed closed and $\lambda$-strategically closed, such that letting $G$ be $\bb{S}$-generic
over $V$, a strengthening of $\square(\lambda,\kappa)$ called
$\square^{\text{ind}}(\lambda,\kappa)$ holds in $V[G]$; see
\cite[Lemma 3.14]{Hayut-LambieHanson:SimultaneousReflectionAndSquare}. We do not
need the precise definition of $\square^{\mathrm{ind}}(\lambda, \kappa)$ here.
It suffices for the present purposes to know that a
$\square^{\text{ind}}(\lambda,\kappa)$-sequence
is in particular a full $\square(\lambda, \kappa)$-sequence. So by Theorem
\ref{thm:SimulReflectionImpliesNoFullSquare}, it follows that
$\refl(\kappa,S)$ fails for every stationary $S\subseteq\lambda$ in $V[G]$.
But by assumption, $\DSR({<}\kappa,S^\lambda_\omega)$ continues to hold.
\end{proof}

\section{Diagonal stationary reflection and the strong reflection principle}
\label{sec:DSRandSRP}

Let us now make a connection to \Todorcevic's strong reflection principle. It will be useful to recall some definitions and facts.

\begin{definition}
\label{def:GeneralizedStationarityNotations}
Let $\gamma$ be a regular uncountable cardinal, and let $X \supseteq \gamma$ be a set.
For a set $A\subseteq\gamma$, let
\[\lift(A,[X]^\omega)=\{x\in[X]^\omega\st\sup(x\cap\gamma)\in A\}.\]
Now let $S\subseteq[X]^\omega$ be stationary. Then $S$ is \emph{projective stationary} if for every stationary set $A\subseteq\omega_1$, the set
\[S\cap\lift(A,[X]^\omega)\]
is stationary. 

If $W\subseteq X\subseteq Y$, then we write
\[S\uparrow[Y]^\omega=\{y\in[Y]^\omega\st y\cap X\in S\}\]
and
\[S\downarrow[W]^\omega=\{x\cap W\st x\in S\}.\]
The \emph{strong reflection principle} (\textsf{SRP}) is the assertion that for every regular cardinal $\theta\ge\omega_2$, if $S\subseteq[H_\theta]^\omega$ is projective stationary, then there is a continuous $\in$-chain $\langle M_i \mid i < \omega_1 \rangle$ of countable elementary submodels of $\langle H_\theta,\in\rangle$ such that for all $i<\omega_1$, $M_i\in S$.
\end{definition}

\begin{fact}
\label{fact:ProjectionAndLifting}
Let $W\subseteq X\subseteq Y$, and let $S\subseteq[X]^\omega$ be stationary.
\begin{enumerate}[label=(\arabic*)]
  \item
  \label{item:LiftingAndProjectingPreservesStationarity}
  $S\downarrow[W]^\omega$ and $S\uparrow[Y]^\omega$ are stationary.
  \item
  \label{item:LiftingAndProjectingPreservesProjectiveStationarity}
  If $\omega_2\subseteq W$ and $S$ is projective stationary, then $S\downarrow[W]^\omega$ and $S\uparrow[Y]^\omega$ are projective stationary.
  \item
  \label{item:IntersectingWithClubs}
  If $S$ is projective stationary and $C\subseteq[X]^\omega$ is club, then $S\cap C$ is projective stationary.
\end{enumerate}
\end{fact}

\begin{proof}
Item \ref{item:LiftingAndProjectingPreservesStationarity} is well-known; see, for
example, \cite{StationaryTower}. For item
\ref{item:LiftingAndProjectingPreservesProjectiveStationarity}, we refer the
reader to \cite[Ex.~37.17]{ST3}. Item \ref{item:IntersectingWithClubs} is immediate:
if $A\subseteq\omega_1$ is stationary, then $(S\cap C)\cap\lift(A,[X]^\omega)=
(S\cap\lift(A,[X]^\omega))\cap C$ is stationary.
\end{proof}

The following lemma is essentially due to Larson \cite{Larson:SeparatingSRP}.
The original statement and proof, though, have a confusing typo and
omit one implication, so we reformulate and prove it here.  We use the notation of Definition \ref{def:KillingStationarySets}; compare to Lemma
\ref{lem:BasicsOnT_S}.

\begin{lemma}
\label{lem:WhenT_APreservesGenStationarity}
Let $\gamma>\omega_1$ be regular, $X\supseteq\gamma$ a set, $A\subseteq \gamma$ stationary and $S\subseteq[X]^\omega$ also stationary, such that $\gamma\setminus A$ is unbounded in $\gamma$ and $\bb{T}_A$ is countably distributive (this is the case, for example, if $A\subseteq S^\gamma_\omega$ and $S^\gamma_\omega \setminus
A$ is stationary). Then the following are equivalent:
\begin{enumerate}[label=(\arabic*)]
\item
\label{item:SminusAisStationary}
$S\setminus\lift(A,[X]^\omega)$ is stationary.
\item
\label{item:StationarityOfSisPreserved}
$\bb{T}_A$ preserves the stationarity of $S$.
\item
\label{item:SomebodyPreservesStationarityOfS}
There is a condition $p\in\bb{T}_A$ that forces that $\check{S}$ is  stationary.
\end{enumerate}
\end{lemma}

\begin{proof}
\ref{item:SminusAisStationary}$\implies$\ref{item:StationarityOfSisPreserved}:
Suppose that \ref{item:SminusAisStationary} holds, and assume towards a contradiction
that \ref{item:StationarityOfSisPreserved} fails.
Let $\tau$ be a $\bb{T}_A$-name that is forced by some condition $p\in\bb{T}_A$
to be a club subset of $[X]^\omega$ disjoint from $S$. By
\ref{item:SminusAisStationary}, $S\setminus\lift(A,[X]^\omega)$ is stationary.
Let $\eta$ be a regular cardinal large enough that $H_\eta$ contains all the relevant parameters.
By Fact \ref{fact:ProjectionAndLifting}, $(S\setminus\lift(A,[X]^\omega))
\uparrow[H_\eta]^\omega$ is stationary, so we can pick a countable elementary submodel $N \prec \langle
H_\eta, \in, p, A, S, \tau \rangle$ such that $N\cap X\in S\setminus\lift(A,[X]^\omega)$. Thus,
$p,A,S,\tau\in N$, $N\cap X\in S$ and $\nu=\sup(N\cap\gamma)\notin A$. Let
$g\subseteq\bb{T}_A\cap N$ be $N$-generic for $\bb{T}_A$, with $p\in g$. Since
$\nu=\sup(\bigcup g)=\sup(N\cap\gamma)\notin A$, it follows that
$q=(\bigcup g)\cup\{\nu\}\in\bb{T}_A$. Also, there is a $\subseteq$-increasing
sequence $\langle x_n \mid n < \omega \rangle$ of elements of $\tau^g$
such that $\bigcup_{n < \omega} x_n
= N \cap X$. But if $q\in G$, where $G$ is $\bb{T}_A$-generic over $V$, then $\tau^g\subseteq\tau^G$, and so $N\cap X\in\tau^G\cap S$, a contradiction.

\ref{item:StationarityOfSisPreserved}$\implies$\ref{item:SomebodyPreservesStationarityOfS}: trivial.

\ref{item:SomebodyPreservesStationarityOfS}$\implies$\ref{item:SminusAisStationary}:
Suppose that \ref{item:SomebodyPreservesStationarityOfS} holds, and assume towards
a contradiction that \ref{item:SminusAisStationary} fails. The failure of
\ref{item:SminusAisStationary} means that there is a club $C\subseteq[X]^\omega$
disjoint from $S\setminus\lift(A,[X]^\omega)$, i.e.,
\begin{enumerate}
\item[$(*)$] $S\cap C\subseteq\lift(A,[X]^\omega)$.
\end{enumerate}
Using \ref{item:SomebodyPreservesStationarityOfS}, let $G\subseteq\bb{T}_A$ be
generic so that $S$ is stationary in $V[G]$. Hence, $S\cap C$ is stationary in $V[G]$.

Let $D=\bigcup G$, so $D\subseteq\gamma$ is club and $D\cap A=\emptyset$. Since $\bb{T}_A$ preserves the fact that $\gamma$ has uncountable cofinality, we can pick $x\in S\cap C$ with $\sup(x\cap\gamma)\in D$. But then $\sup(x\cap\gamma)\notin A$, which contradicts $(*)$.
\end{proof}

It follows from \cite[Theorem 4.6]{Larson:SeparatingSRP} that, assuming the consistency of Martin's Maximum, {\sf SRP} does not imply $\OSR(S^{\omega_2}_\omega)$, which by Lemma \ref{equivalence_lemma} is equivalent to each of $\DSR(\omega_1,S^{\omega_2}_\omega)$,
$\sDSR(\omega_1,S^{\omega_2}_\omega)$, and $\uDSR(\omega_1,S^{\omega_2}_\omega)$.
We now show that one can do slightly better at regular cardinals $\lambda > \omega_2$ by adapting the proof of Theorem \ref{thm:SimultaneousReflDoesNotImplyUnbounded}
to prove that \textsf{SRP} does not imply $\uDSR(1, S^\lambda_\omega)$. The hypothesis of the following theorem follows from Martin's Maximum.

\begin{theorem}
\label{thm:SeparatingSRPfromuDSR}
Let $\lambda > \omega_2$ be a regular cardinal, and suppose that
{\sf SRP} holds and continues to hold in any forcing extension obtained by a $\lambda$-directed closed forcing notion. Then there is a $\lambda$-strategically closed forcing notion which produces forcing extensions in which
\begin{enumerate}
  \item {\sf SRP} continues to hold, but
  \item $\uDSR(1,S^\lambda_\omega)$ fails.
\end{enumerate}
\end{theorem}

\begin{proof}
We use the forcing notion $\bb{P}$ of Theorem \ref{thm:SimultaneousReflDoesNotImplyUnbounded}. Let $G$ be $\bb{P}$-generic over $V$. It was shown in the proof of that theorem that $\bb{P}$ is $\lambda$-strategically closed and adds a sequence $\langle S_\alpha\st\alpha<\lambda\rangle$ of stationary subsets of $S^\lambda_\omega$ which is a counterexample to $\uDSR(1,S^\lambda_\omega)$. Recall that, for $\alpha<\lambda$, we set $S_{{\ge}\alpha}=\bigcup_{\alpha\le\beta<\lambda}S_\beta$ and let
$\dot{\bb{T}}_\alpha$ be a canonical name for the forcing $\bb{T}_{S_{{\ge}\alpha}}$. We have seen that $\bb{P}*\dot{\bb{T}}_\alpha$ has a dense $\lambda$-directed closed subset. Other crucial properties of the sequence $\vec{S}$, which we will use in the proof of
the following claim, are that for $\alpha<\lambda$, $S_\alpha\cap(\alpha+1)=\emptyset$ and that for $\alpha<\beta<\lambda$, $S_\alpha\cap S_\beta=\emptyset$. Work now in $V[G]$.

\begin{claim}
\label{claim:PreservingStationarity}
Let $\theta\ge\lambda$ be a regular cardinal, and let $S\subseteq[H_\theta]^\omega$ be stationary. Then there is an $\alpha<\lambda$ such that $\bb{T}_\alpha$ preserves the stationarity of $S$.
\end{claim}

\begin{proof}
If not, then, by Lemma \ref{lem:WhenT_APreservesGenStationarity}, for every
$\alpha<\lambda$, the set
\[
  B_\alpha=\{x\in S\st\sup(x\cap\lambda)\notin S_{{\ge}\alpha}\}
\]
is not stationary. Thus, for all $\alpha < \lambda$, let $C_\alpha\subseteq[H_\theta]^\omega$ be a club disjoint
from $B_\alpha$. By normality of the generalized club filter, there is a club
$D\subseteq[H_\theta]^\omega$ with the property that
\[\forall x\in D\forall\xi\in x\cap\lambda\quad x\in C_\xi.\]
Now let $x\in S\cap D$, and set $\sigma=\sup(x\cap\lambda)$. For all $\xi\in x\cap\lambda$, $x\in C_\xi$. So since $x\in S$ and $C_\xi\cap B_\xi=\emptyset$, it follows that $\sigma\in S_{{\ge}\xi}$. So, since $x\cap\lambda$ is cofinal in $\sigma$ and $\vec{S}$ is pairwise disjoint, it follows that $\sigma\in S_{{\ge}\sigma}$. But for $\zeta<\lambda$ with $\zeta\ge\sigma$, we have $S_\zeta\cap(\sigma+1)\subseteq S_\zeta\cap(\zeta+1)=\emptyset$, so it cannot be that $\sigma\in S_{{\ge}\sigma}$.
\end{proof}

Next, we are going to show that {\sf SRP} holds in $V[G]$. So, still working in $V[G]$, let $S\subseteq[H_\theta]^\omega$ be projective stationary, where $\theta\ge\omega_2$ is a regular cardinal.

\begin{claim}
\label{claim:PreservingProjectiveStationarity}
There is an $\alpha<\lambda$ such that $\bb{T}_\alpha$ preserves the projective stationarity of $S$.
\end{claim}

\begin{proof}
First, note that we may assume that $\theta\ge\lambda$. For otherwise, we can choose $\theta'\ge\lambda>\theta$ and set $S'=S\uparrow[H_{\theta'}]^\omega$. By Fact \ref{fact:ProjectionAndLifting}, $S'$ is projective stationary. If we can show that the projective stationarity of $S'$ is preserved by some $\bb{T}_\alpha$, then by applying Fact \ref{fact:ProjectionAndLifting} to $S=S'\downarrow[H_\theta]^\omega$, it follows that $S$ is projective stationary in the extension.

So let us assume that $\theta\ge\lambda$. For every stationary set $A\subseteq\omega_1$, let
\[S_A=S\cap\lift(A,[H_\theta]^\omega).\]
Using Claim \ref{claim:PreservingStationarity}, let $\alpha_A<\lambda$ be such that $\bb{T}_{\alpha_A}$ preserves the stationarity of $S_A$. Let \[\alpha=\sup\{\alpha_A\st A\ \text{is a stationary subset of}\ \omega_1\}.\]
Recall that {\sf SRP} implies that $2^{\omega_1}=\omega_2$ (this is due to \Todorcevic;
cf.\ \cite[Theorem 9.82]{Woodin:ADforcingAxiomsNonstationaryIdeal}). Since {\sf SRP} holds
in $V$ and $\bb{P}$ is $\lambda$-strategically closed, this consequence still holds in
$V[G]$. So we have that in $V[G]$, $2^{\omega_1}<\lambda$, which implies that $\alpha<\lambda$.

It now follows from Lemma \ref{lem:WhenT_APreservesGenStationarity} that $\bb{T}_\alpha$
preserves the stationarity of $S_A$ for every stationary $A \subseteq \omega_1$ in
$V[G]$. This is because by that lemma, it suffices to show that $S_A\setminus\lift(S_{{\ge}\alpha},[H_\theta]^\omega)$ is stationary. But we know that $\bb{T}_{\alpha_A}$ preserves the stationarity of $S_A$, which again by the lemma means that $S_A\setminus\lift(S_{{\ge}\alpha_A},[H_\theta]^\omega)$ is stationary. Since $\alpha_A\le\alpha$, we have that $S_{\alpha}\subseteq S_{\alpha_A}$, and so $S_A\setminus\lift(S_{{\ge}\alpha_A},[H_\theta]^\omega)\subseteq
S_A\setminus\lift(S_{{\ge}\alpha},[H_\theta]^\omega)$ is stationary, as wished.

Since, in $V$, $\bb{P} * \dot{\bb{T}}_\alpha$ has a dense $\lambda$-directed closed
subset, we know that, in $V[G]$, $\bb{T}_\alpha$ is $\lambda$-distributive and hence
does not add any new subsets of $\omega_1$. In particular, every stationary subset of $\omega_1$
in the extension by $\bb{T}_\alpha$ is already in $V[G]$. Therefore, forcing
with $\bb{T}_\alpha$ preserves the projective stationarity of $S$.
\end{proof}

We can now finish the proof by showing that, in $V[G]$, $S$ contains a continuous
$\in$-chain of length $\omega_1$. Let $\alpha<\lambda$ be such that $\bb{T}_\alpha$
preserves the projective stationarity of $S$, and let $H$ be $\bb{T}_\alpha$-generic
over $V[G]$. In $V[G * H]$, $S$ is projective stationary. Working in $V[G * H]$,
let $\theta'\ge\theta$ be a regular cardinal. We can form $S'=S\uparrow
([H_{\theta'}]^\omega)^{V[G * H]}$. $S'$ is then projective stationary in $V[G * H]$
by Fact \ref{fact:ProjectionAndLifting}; note that $\theta \geq \omega_2$ in $V[G * H]$.
Moreover, by the same fact, $S'\cap C$ is also projective stationary, whenever
$C\subseteq[H_{\theta'}]^\omega$ is a club. In particular, the set
\[
  T=\{x\in S'\st x\prec\langle H_{\theta'}^{V[G * H]},\in,H_\theta^{V[G]}\rangle\}
\]
is projective stationary in $V[G * H]$.

Since $\bb{P}*\dot{\bb{T}}$ is equivalent to a $\lambda$-directed closed forcing
in $V$, it follows by our assumptions that {\sf SRP} holds in $V[G * H]$. Hence,
there is an $\omega_1$-chain $\langle N_i\st i<\omega_1\rangle$ of elementary
submodels of $H_\theta^{V[G * H]}$ in $V[G * H]$ such that, for every $i<\omega_1$,
$N_i\in T$. Set $M_i=N_i\cap H_\theta^{V[G]}$, for $i<\omega_1$. Then, since
$\bb{T}_\alpha$ is $\lambda$-distributive in $V[G]$, it
follows that the sequence $\langle M_i\st i<\omega_1\rangle$ is in $V[G]$.
Moreover, it is a continuous $\in$-chain and, for every $i<\omega_1$, $M_i$
is an elementary submodel of $H_\theta^{V[G]}$
(since $H_\theta^{V[G]}\cap N_i$ is available as a predicate in $N_i$) and $M_i\in S$. This proves this instance of
{\sf SRP}.
\end{proof}

In order to use the argument of the previous proof to obtain the failure of $\uDSR(1,S^{\omega_2}_\omega)$ while ${\sf SRP}$ holds, we seem to need a stronger assumption. In fact, we do not know whether its consistency follows from any large cardinal assumption. Larson \cite[Remark before Def.~6.5]{Larson:SeparatingSRP} points out that it follows from results of Woodin \cite{Woodin:ADforcingAxiomsNonstationaryIdeal} that one can derive a model in which $\mathsf{SRP}(\omega_2)$ holds and the nonstationary ideal on $\omega_1$ has density $\omega_1$ from a model of $\mathsf{AD}_{\mathbb{R}} + $ ``$\theta$ is regular''. This is at least going in the direction of our assumption.

\begin{corollary}
\label{cor:SeparatingSRPfromuDSR}
Suppose that
{\sf SRP} holds and continues to hold in any forcing extension obtained by an $\omega_2$-directed closed forcing notion. Assume furthermore that the density of the nonstationary ideal on $\omega_1$ is $\omega_1$. Then
there is an $\omega_2$-strategically closed forcing notion which produces forcing extensions where
\begin{enumerate}
  \item {\sf SRP} continues to hold, but
  \item $\uDSR(1,S^{\omega_2}_\omega)$ fails.
\end{enumerate}
\end{corollary}

\begin{proof}
By assumption, we may fix an $\omega_1$-sized collection $\mathcal{A}$ of stationary subsets of $\omega_1$ which is dense in the stationary subsets of $\omega_1$, that is, for every stationary $B\subseteq\omega_1$, there is an $A\in\mathcal{A}$ such that $A\subseteq B$.

We now argue as in the proof of Theorem \ref{thm:SeparatingSRPfromuDSR}, with $\lambda=\omega_2$. So we let $G$ be generic for $\bb{P}$, adding a sequence $\langle S_\alpha \st \alpha<\lambda\rangle$ of stationary subsets of $S^{\omega_2}_\omega$ which forms a counterexample to $\uDSR(1,S^{\omega_2}_\omega)$. We define $S_{{\ge}\alpha}$ and $\dot{\bb{T}}_\alpha$ as before, for $\alpha<\omega_2$. As before, we see that in $V[G]$, for every stationary $S\subseteq [H_\theta]^\omega$, where $\theta\ge\omega_2$, there is an $\alpha<\omega_2$ such that $\bb{T}_\alpha$ preserves the stationarity of $S$. We can now follow the proof of Claim \ref{claim:PreservingProjectiveStationarity} to show that in $V[G]$, if $S$ is projective stationary in $[H_\theta]^\omega$, where $\theta\ge\omega_2$ is regular, then there is an $\alpha<\omega_2$ such that $\bb{T}_\alpha$ preserves the projective stationarity of $S$. To see this, let, for any stationary $B\subseteq\omega_1$, $S_B=S\cap\lift(B,[H_\theta]^\omega)$ and choose, an ordinal $\alpha_B<\omega_2$ such that $\bb{T}_{\alpha_A}$ preserves the stationarity of $S_A=S\cap\lift(A,[H_\theta]^\omega)$. Since the cardinality of $\mathcal{A}$ is $\omega_1$, we know that $\alpha=\sup\{\alpha_A\st A\in\mathcal{A}\}<\omega_2$. It then follows that $\bb{T}_\alpha$ preserves the stationarity of $S_B$, for every stationary $B\subseteq\omega_1$. Namely, given such a $B$, it suffices to show that $S_B\setminus\lift(S_{{\ge}\alpha},[H_\theta]^\omega)$ is stationary. But by density of $\mathcal{A}$, there is an $A\in\mathcal{A}$ with $A\subseteq B$, and we know that $S_A\setminus\lift(S_{{\ge}\alpha_A},[H_\theta]^\omega)$ is stationary. Clearly, $S_A\subseteq S_B$, and $S_{{\ge}\alpha}\subseteq S_{{\ge}\alpha_A}$, so that
$S_A\setminus\lift(S_{{\ge}\alpha_A},[H_\theta]^\omega)\subseteq S_B\setminus\lift(S_{{\ge}\alpha},[H_\theta]^\omega)$ is stationary.

Since $\bb{P*}\dot{\bb{T}}_\alpha$ has an $\omega_2$-directed closed dense subset, {\sf SRP} holds in $V^{\bb{P}*\dot{\bb{T}}_\alpha}$, where $S$ is projective stationary, and this goes down to $V[G]$ as before.
\end{proof}

\section{Indexed square and $\sDSR$} \label{sec:Square}

In this section, we will show that Theorem \ref{sdsr_square_incompatibility_theorem}
is sharp by constructing models in which both $\sDSR({<}\kappa, S)$ and
$\square(\lambda, \kappa)$ hold, where $\kappa < \lambda$ are infinite regular
cardinals and $S \subseteq \lambda$ is stationary. Notice that we have in fact
already done this for $S = S^\lambda_\omega$ in the proof of Theorem
\ref{DSR_simultaneous_thm} (and in fact we obtained $\DSR({<}\kappa, S^\lambda_\omega)$
there). Our reasons for including this section are twofold. Firstly, we can
significantly reduce the large cardinals necessary. The hypotheses of Theorem
\ref{DSR_simultaneous_thm} can be obtained by assuming, for instance, that
$\kappa$ is indestructibly generically supercompact, whereas
we can achieve the hypotheses of Theorem \ref{sdr_square_thm} by starting in
an inner model in which $\lambda$ is weakly compact (at least, if we want $\lambda$
to be either inaccessible or the successor of a regular cardinal in the final
model). Secondly, Theorem \ref{sdr_square_thm} is more general in the sense
that it applies to any stationary subset $S \subseteq \lambda$.

We will need the following strengthening of $\square(\lambda, \kappa)$, introduced in
\cite{lh_narrow_systems}.

\begin{definition}
\label{def:IndexedSquare}
  Suppose that $\kappa < \lambda$ are infinite regular cardinals. A
  $\square^{\mathrm{ind}}(\lambda, \kappa)$ sequence is a matrix $\vec{\mathcal{C}}
  = \langle C_{\alpha, i} \mid \alpha \in \lim(\lambda), ~ i(\alpha) \leq i < \kappa \rangle$
  satisfying the following conditions.
  \begin{enumerate}
    \item For all $\alpha \in \lim(\lambda)$, we have $i(\alpha) < \kappa$.
    \item For all $\alpha \in \lim(\lambda)$ and $i(\alpha) \leq i < \kappa$,
    $C_{\alpha, i}$ is a club in $\alpha$.
    \item For all $\alpha \in \lim(\lambda)$ and $i(\alpha) \leq i < j < \kappa$,
    we have $C_{\alpha, i} \subseteq C_{\alpha, j}$.
    \item For all limit $\alpha, \beta \in \lim(\lambda)$ and $i(\beta) \leq i < \kappa$,
    if $\alpha \in \lim(C_{\beta, i})$, then $i(\alpha) \leq i$ and
    $C_{\alpha, i} = C_{\beta, i} \cap \alpha$.
    \item For all $\alpha, \beta \in \lim(\lambda)$ with $\alpha < \beta$, there is $i$ with
    $i(\beta) \leq i < \kappa$ such that $\alpha \in \lim(C_{\beta, i})$.
    \item There is no club $D \subseteq \lambda$ such that, for every $\alpha
    \in \lim(D)$, there is $i$ with $i(\alpha) \leq i < \kappa$ such that
    $D \cap \alpha = C_{\alpha, i}$. (Such a club would be a \emph{thread}
    through $\vec{\mathcal{C}}$.)
  \end{enumerate}
  $\square^{\mathrm{ind}}(\lambda, \kappa)$ is the assertion that there is a
  $\square^{\mathrm{ind}}(\lambda, \kappa)$-sequence.
\end{definition}

It is clear from the definition that $\square^{\mathrm{ind}}(\lambda, \kappa)$
implies the existence of a full $\square(\lambda, \kappa)$-sequence (see Definition \ref{def:Full}. Namely, by condition (5) if we let $\mathcal{C}_\alpha=\{C_{\alpha,i}\st i(\alpha)\le i<\kappa\}$, for limit $\alpha<\lambda$, then $\vec{\mathcal{C}}$ is a $\square(\lambda,\kappa)$-sequence with the property that for all $\alpha, \beta \in \lim(\lambda)$ with $\alpha<\beta$, $\alpha$ is a limit point of some $C\in\mathcal{C}_\beta$. This is a much stronger property than fullness.

\begin{theorem}
  \label{sdr_square_thm}
  Suppose that $\kappa < \lambda$ are infinite regular cardinals, $\lambda^{<\lambda} = \lambda$,
  $S \subseteq \lambda$ is stationary, and $\DSR^*({<}\kappa, S)$ holds. Then there is a
  cofinality-preserving forcing extension in which $S$ remains stationary and
  both $\sDSR({<}\kappa, S)$ and $\square^{\mathrm{ind}}(\lambda, \kappa)$ hold.
\end{theorem}

\begin{remark}
  Note that, since $\square^{\mathrm{ind}}(\lambda, \kappa)$ implies the
  existence of a full $\square(\lambda, \kappa)$ sequence, we can conclude
  from Theorem \ref{thm:SimulReflectionImpliesNoFullSquare} that
  $\refl(\kappa, T)$ fails for every stationary $T \subseteq \lambda$ in the
  model obtained in Theorem \ref{sdr_square_thm}. Therefore, this gives
  an alternate proof of Theorem \ref{stationary_simultaneous_thm} in the cases
  in which $\kappa \geq \omega$.
\end{remark}

\begin{proof}[Proof of Theorem \ref{sdr_square_thm}]
  By the results of \cite[Section 3.2]{Hayut-LambieHanson:SimultaneousReflectionAndSquare},
  there is a two-step forcing iteration $\bb{S} * \dot{\bb{P}}$ with the following
  salient properties.
  \begin{itemize}
    \item $\bb{S}$ has cardinality $\lambda$ and, in $V^{\bb{S}}$, $\bb{P}$ is
    a forcing iteration of length $\lambda^+$, taken with supports of size
    less than $\lambda$, in which each iterand has cardinality $\lambda$.
    $\bb{P}$ therefore has the $\lambda^+$-cc in $V^{\bb{S}}$.
    For $\eta \leq \lambda^+$, let $\bb{P}_\eta$ denote the initial segment of
    length $\eta$ of this iteration.
    \item In $V^{\bb{S} * \dot{\bb{P}}}$, $\square^{\mathrm{ind}}(\lambda, \kappa)$
    holds, as witnessed by a sequence $\vec{\mathcal{C}} = \langle C_{\alpha, i}
    \mid \alpha < \lambda, ~ i(\alpha) \leq i < \kappa \rangle$ explicitly
    introduced by $\bb{S}$.
    \item In $V^{\bb{S}}$, for each $i < \kappa$, define a forcing poset
    $\bb{T}_i$ as follows. Conditions of $\bb{T}_i$ are all clubs $C_{\alpha, i}$
    (from the $\square^{\mathrm{ind}}(\lambda, \kappa)$-sequence isolated above)
    such that $i(\alpha) \leq i$. If $C_{\alpha, i}$ and $C_{\beta, i}$ are in
    $\bb{T}_i$, then $C_{\beta, i} \leq_{\bb{T}_i} C_{\alpha, i}$ if and only
    if $C_{\alpha, i} = C_{\beta, i} \cap \alpha$. ($\bb{T}_i$ is the forcing
    to add a thread through the $i^{\mathrm{th}}$ column of $\vec{\mathcal{C}}$.)
    Then the following hold.
    \begin{itemize}
      \item In $V$, for all $i < \kappa$ and all $\eta \leq \lambda^+$,
      $\bb{S} * \dot{\bb{P}}_\eta * \dot{\bb{T}}_i$ has a dense $\lambda$-directed
      closed subset. Moreover, if $\eta < \lambda^+$, then this subset has
      cardinality $\lambda$.
      \item In $V^{\bb{S} * \dot{\bb{P}}}$, for all $i < j < \kappa$, the map
      $C_{\alpha, i} \mapsto C_{\alpha, j}$ is a projection from $\bb{T}_i$ to
      $\bb{T}_j$. This projection will be denoted by $\pi_{ij}:\bb{T}_i
      \rightarrow \bb{T}_j$.
      \item In $V^{\bb{S} * \dot{\bb{P}}}$, if $T$ is a stationary subset of
      $\lambda$, then there are $i < \kappa$ and $t \in \bb{T}_i$ such that
      $t \Vdash_{\bb{T}_i} `` \check{T} \text{ is stationary}"$.
    \end{itemize}
  \end{itemize}

  Let $G * H$ be $\bb{S} * \dot{\bb{P}}$-generic over $V$. For $\eta \leq \lambda^+$,
  let $H_\eta$ be the $\dot{P}_\eta$-generic filter induced by $H$. $V[G * H]$
  is our desired model. Notice that since, in $V$, $\bb{S} * \dot{\bb{P}} *
  \dot{\bb{T}}_0$ has a dense $\lambda$-directed closed subset, it preserves all
  stationary subsets of $\lambda$, and hence $S$ remains stationary in
  $V[G * H]$. It remains to verify that $\sDSR({<}\kappa, S)$ holds
  in $V[G * H]$.

  To this end, work in $V[G * H]$ and suppose that $\langle S_{\alpha, j} \mid \alpha < \lambda, ~ j <
  j_\alpha \rangle$ is a matrix of stationary subsets of $S$, where
  $j_\alpha < \kappa$ for every $\alpha < \lambda$. We will find
  $\gamma \in S^\lambda_{>\omega}$ such that, for stationarily many $\alpha < \gamma$,
  for all $j < j_\alpha$, we have that $S_{\alpha, j} \cap \gamma$ is stationary in
  $\gamma$.

  By the final property of $\bb{S} * \dot{\bb{P}}$ listed above,
  for each $\alpha < \lambda$ and
  $j < j_\alpha$, we
  can find an $i_{\alpha, j} < \kappa$ and a condition $t_{\alpha, j} \in
  \bb{T}_{i_{\alpha, j}}$ such that $t_{\alpha, j} \Vdash_{\bb{T}_{i_{\alpha, j}}}
  ``\check{S}_{\alpha, j} \text{ is stationary}"$. Notice that, for each
  such $\alpha$ and $j$ and all $k$ with $i_{\alpha, j} < k < \kappa$,
  we also have $\pi_{i_{\alpha, j}k}(t_{\alpha, j}) \Vdash_{\bb{T}_k}
  ``\check{S}_{\alpha, j} \text{ is stationary}"$ by the arguments of Claim
  \ref{claim_323}.

  For each $\alpha < \lambda$ and $j < j_\alpha$, let $\beta_{\alpha, j}$ be
  such that $t_{\alpha, j} = C_{\beta_{\alpha, j}, i_{\alpha, j}}$. For
  each $\alpha < \lambda$, find a limit ordinal $\beta_\alpha$ with
  $\alpha < \beta_\alpha < \lambda$ such
  that $\beta_{\alpha, j} < \beta_\alpha$ for all $j < j_\alpha$. Since
  $j_\alpha < \kappa$, the definition of $\square^{\mathrm{ind}}(\lambda, \kappa)$
  implies that we can find an ordinal $i_\alpha$ with $i(\beta_\alpha) \leq
  i_\alpha < \kappa$ such that, for all $j < j_\alpha$, we have
  $\beta_{\alpha, j} \in \lim(C_{\beta_\alpha, i_\alpha})$ and $i_\alpha > i_{\alpha, j}$.
  Letting $t_\alpha = C_{\beta_\alpha, i_\alpha}$, it follows that, for every
  $j < j_\alpha$, we have $t_\alpha \leq_{\bb{T}_{i_\alpha}} \pi_{i_{\alpha, j}
  i_\alpha}(t_{\alpha, j})$, and hence $t_\alpha \Vdash_{\bb{T}_{i_\alpha}}
  ``\check{S}_{\alpha, j} \text{ is stationary}"$. For $i$ with $i_\alpha
  \leq i < \kappa$, let $t_\alpha^i = \pi_{i_\alpha i}(t_\alpha)$.

  By the chain condition of $\bb{P}$, we can find $\eta < \lambda^+$ such that
  \[
    \langle S_{\alpha, j} \mid \alpha < \lambda, ~ j < j_\alpha \rangle \in
    V[G * H_\eta].
  \]
  Work now in $V[G * H_\eta]$. Since stationarity is downward absolute,
  it is still the case in $V[G * H_\eta]$
  that, for all $\alpha < \lambda$, all $j < j_\alpha$, and all $i$ with
  $i_\alpha \leq i < \kappa$, we have $t^i_\alpha \Vdash_{\bb{T}_i} ``
  \check{S}_{\alpha, j} \text{ is stationary}"$.

  For each $i < \kappa$, let $\dot{J}_i$ be the $\bb{T}_i$-name for the generic filter,
  and let $\dot{R}_i$ be the $\bb{T}_i$-name for the set
  \[
    \{\alpha \in S \mid i_\alpha \leq i \text{ and } t^i_\alpha \in \dot{J}_i\}.
  \]

  \begin{claim}
    There is $i < \kappa$ and $r \in \bb{T}_i$ such that
    \[
      r \Vdash_{\bb{T}_i}``\dot{R}_i \text{ is stationary}".
    \]
  \end{claim}

  \begin{proof}
    Suppose not. Then, for each $i < \kappa$, there is a $\bb{T}_i$-name $\dot{E}_i$
    for a club in $\lambda$ disjoint from $\dot{R}_i$. Since, for each $i < \kappa$,
    $\pi_{0i}$ is a projection from $\bb{T}_0$ to $\bb{T}_i$, each $\dot{E}_i$
    can be interpreted as a $\bb{T}_0$-name, so we can let $\dot{E}$ be a
    $\bb{T}_0$-name for $\bigcap_{i < \kappa} \dot{E}_i$.

    Let $J$ be $\bb{T}_0$-generic over $V[G * H_\eta]$, let $D = \bigcup J$ be
    the generic club added by $J$ that threads $\vec{C}$, and let $E$ be the interpretation
    of $\dot{E}$ in $V[G * H_\eta * J]$. Notice that, since $S$ is a stationary
    subset of $\lambda$ in $V$ and $\bb{S} * \dot{\bb{P}}_\eta * \dot{\bb{T}}_0$
    has a dense $\lambda$-directed closed subset, $S$ remains stationary in
    $V[G * H_\eta * J]$.

    For each $\gamma \in E$, find an ordinal $\xi_\gamma$ with $\gamma \leq
    \xi_\gamma < \lambda$ such that
    $C_{\xi_\gamma, 0} \in J$ and, in $V[G * H_\eta]$, $C_{\xi_\gamma, 0}
    \Vdash_{\bb{T}_0} ``\check{\gamma} \in \dot{E}"$. Note that, by our definition
    of $\dot{E}$, it follows that, for all $i < \kappa$, we have
    $C_{\xi_\gamma, i} \Vdash_{\bb{T}_i} ``\check{\gamma} \in \dot{E}_i"$.
    Let
    \[
      E^* = \{\delta \in \lim(E) \mid \text{for all }\gamma \in E \cap \delta,
      \text{ we have } \xi_\gamma < \delta\}.
    \]
    Then $E^*$ is a club in $\lambda$ and $E^* \subseteq \lim(D)$. We can therefore
    find $\delta \in E^* \cap S$. Then $C_{\delta, 0} \in J$ and
    $C_{\delta,0} \leq_{\bb{T}_0} C_{\xi_\gamma, 0}$ for every $\gamma \in
    E \cap \delta$. Since $\sup(E \cap \delta) = \delta$ and, in
    $V[G * H_\eta]$, $\dot{E}$ is forced to be a club, we know that
    $C_{\delta, 0} \Vdash_{\bb{T}_0} ``\check{\delta} \in \dot{E}."$ in
    $V[G * H_\eta]$. By the definition of $\dot{E}$, it follows that, for all
    $i < \kappa$, we have $C_{\delta, i} \Vdash_{\bb{T}_i}``\check{\delta} \in
    \dot{E}_i"$.

    Recall that we previously found a limit ordinal $\beta_\delta$ with
    $\delta < \beta_\delta < \lambda$ and an ordinal $i_\delta < \kappa$ such
    that, for all $i$ with $i_\delta \leq i < \kappa$, we have $t^i_\delta =
    C_{\beta_\delta, i}$. Let $i^* < \kappa$ be least such that
    $i^* \geq i_\delta$ and $\delta \in \lim(C_{\beta_\delta}, i^*)$.
    Then $t^i_\delta \leq_{\bb{T}_i} C_{\delta, i}$ and, clearly,
    $t^i_\delta \Vdash_{\bb{T}_i} ``\check{\delta} \in \dot{R}_i"$. However,
    this contradicts the fact that $C_{\delta, i} \Vdash_{\bb{T}_i} ``\check{\delta}
    \in \dot{E}_i"$ and $\dot{E}_i$ is forced to be disjoint from $\dot{R}_i$.
  \end{proof}

  Choose $i < \kappa$ and $r \in \bb{T}_i$ as in the statement of the claim, and let
  $J$ be $\bb{T}_i$-generic over $V[G * H_\eta]$ with $r \in J$. Let $R$ be
  the interpretation of $\dot{R}_i$ in $V[G * H_\eta * J]$. Note that, for all
  $\alpha \in R_i$ and all $j < j_\alpha$, we know that $S_{\alpha, j}$ remains stationary
  in $V[G * H_\eta * J]$. Since, in $V$,
  $\bb{S} * \dot{\bb{P}}_\eta * \dot{\bb{T}}_i$ has a dense $\lambda$-directed
  closed subset of cardinality $\lambda$ and $\DSR^*({<}\kappa, S)$ holds, it
  follows that $\DSR({<}\kappa, S)$ holds in $V[G * H_\eta * J]$. Working
  in $V[G * H_\eta * J]$, define a matrix
  $\langle \hat{S}_{\alpha, j} \mid \alpha < \lambda, ~ j < j_\alpha + 1 \rangle$
  as follows. For all $\alpha \in R_i$, let $\hat{S}_{\alpha, j} = S_{\alpha, j}$
  for all $j < j_\alpha$ and $\hat{S}_{\alpha, j_\alpha} = R$. For
  $\alpha \in \lambda \setminus R_i$, simply let $\hat{S}_{\alpha, j} = R$
  for all $\alpha < j_\alpha + 1$.

  By $\DSR({<}\kappa, S)$, we can find an ordinal $\gamma \in S^\lambda_{>\omega}$
  and a club $F \subseteq \gamma$ such that $\hat{S}_{\alpha, j} \cap \gamma$
  is stationary in $\gamma$ for every $\alpha \in F$ and every $j < j_\alpha + 1$.
  Since $\hat{S}_{\alpha, j_\alpha} = R$ for every $\alpha < \lambda$, it follows
  that $F \cap R$ is stationary in $\gamma$ and, for all $\alpha \in F \cap R$
  and all $j < j_\alpha$, we have that $S_{\alpha, j} \cap \gamma$ is stationary
  in $\gamma$. Since stationarity is downward absolute, it follows that, in
  $V[G * H_\eta]$, the set of $\alpha < \gamma$ such that $S_{\alpha, j} \cap
  \gamma$ is stationary in $\gamma$ for every $j < j_\alpha$ is itself
  stationary in $\gamma$. Since $V[G * H]$ has the same bounded subsets of
  $\lambda$ as $V[G * H_\eta]$, this continues to hold in $V[G * H]$ as well.
  Therefore, $\gamma$ witnesses this instance of $\sDSR({<}\kappa, S)$ in
  $V[G * H]$.
\end{proof}

\section{Questions}
\label{sec:Questions}

We end with a few questions that remain open.
First, recall Lemma \ref{dsr_implies_simultaneous}, stating that
if $\lambda$ is a regular uncountable cardinal, $S\subseteq\lambda$ is stationary and $\DSR(1,S)$ holds, then $\refl(\omega,S)$ follows. On the other hand, by Theorem \ref{DSR_simultaneous_thm}, it is consistent to have regular cardinals $\omega_1<\kappa<\lambda$ such that $\DSR({<}\kappa,S^\lambda_\omega)$ holds yet $\refl(\kappa,S)$ fails, for any set $S$ stationary in $\lambda$. So the question that remains in this context is:

\begin{question}
Is it consistent that there is a regular cardinal $\lambda>\omega_1$ such that $\DSR(1,S^\lambda_\omega)$ holds but $\refl(\omega_1,\lambda)$ fails? Or that $\DSR(\omega,S^\lambda_\omega)$ holds but $\refl(\omega_1,\lambda)$ fails?
\end{question}

In another direction, recall Theorem \ref{sdsr_square_incompatibility_theorem}, which states that if $1<\kappa<\lambda$, $\lambda$ regular and $\sDSR({<}\kappa,S)$ holds, for some set $S$ stationary in $\lambda$, then $\square(\lambda,{<}\kappa)$ fails. This is an improvement of the original observation, which drew the same conclusion from the assumption of $\DSR({<}\kappa,S)$. We have shown that this is optimal in some sense (see Theorems \ref{DSR_simultaneous_thm} and \ref{sdr_square_thm}), but it is open whether it is optimal in another sense, namely, the following is unknown.

\begin{question}
Suppose $1<\kappa<\lambda$, $\lambda$ regular, and that $\uDSR({<}\kappa,S)$ holds, for some set $S$ stationary in $\lambda$. Does it follow that $\square(\lambda,{<}\kappa)$ fails?
\end{question}

Finally, we ask whether Theorem \ref{thm:sDSRDoesNotImplyDSR} can be improved
to cover the case in which $\lambda = \kappa^+$.

\begin{question}
Is it consistent that $\kappa$ is an uncountable cardinal, $\lambda = \kappa^+$,
and $\sDSR({<}\kappa,S^\lambda_\omega)$ holds but $\DSR(1,S^\lambda_\omega)$ fails?
\end{question}

\bibliographystyle{amsplain}
\bibliography{bib}

\end{document}